\documentclass[11pt]{nyjm}
\usepackage{xcolor,graphicx}
\usepackage{hyperref}
\usepackage{pifont}
\usepackage{amsthm}
\usepackage{amsfonts}
\usepackage{amssymb,amsmath}
\usepackage{upgreek}
\usepackage[all]{xy}
\usepackage{amsmath}
\usepackage{epsfig}

\newtheorem{theorem}{Theorem}[subsection]
\newtheorem{lemma}[theorem]{Lemma}

\newtheorem{corollary}[theorem]{Corollary}
\theoremstyle{definition}
\newtheorem{definition}[theorem]{Definition}
\newtheorem{example}[theorem]{Example}
\theoremstyle{remark}
\newtheorem{remark}[theorem]{Remark}

\begin{document}
\title{Relating virtual knot invariants to links in the $3$-sphere}
\author{Micah Chrisman}
\address{Department of Mathematics, Monmouth University, West Long Branch, New Jersey, USA}
\email{mchrisma@monmouth.edu}
\author{Robert G. Todd}
\address{Department of Natural and Applied Sciences, Mount Mercy University, Cedar Rapids, Iowa, USA}
\email{rtodd@mtmercy.edu}
\begin{abstract} Geometric interpretations of some virtual knot invariants are given in terms of invariants of links in $\mathbb{S}^3$. Alexander polynomials of almost classical knots are shown to be specializations of the multi-variable Alexander polynomial of certain two-component boundary links of the form $J \sqcup K$ with $J$ a fibered knot. The index of a crossing, a common ingredient in the construction of virtual knot invariants, is related to the Milnor triple linking number of certain three-component links $J \sqcup K_1 \sqcup K_2$ with $J$ a connected sum of trefoils or figure-eights. Our main technical tool is virtual covers. This technique, due to Manturov and the first author, associates a virtual knot $\upsilon$ to a link $J \sqcup K$, where $J$ is  fibered and $\text{lk}(J,K)=0$. Here we extend virtual covers to all multicomponent links $L=J \sqcup K$, with $K$ a knot. It is shown that an unknotted component $J_0$ can be added to $L$ so that $J_0 \sqcup J$ is fibered and $K$ has algebraic intersection number zero with a fiber of $J_0 \sqcup J$. This is called fiber stabilization. It provides an avenue for studying all links with virtual knots. 
\end{abstract}
\keywords{virtual knots, virtual covers, multi-variable Alexander polynomial, boundary links, index polynomial, Milnor triple-linking number}
\subjclass[2010]{Primary: 57M25, Secondary: 57M27}

\maketitle

\section{Introduction}

As L. H. Kauffman notes in the preface to ``Virtual Knots: The State of the Art" by Manturov and Ilyutko \cite{vktsoa}, virtual knots were developed to ``simultaneously have a diagrammatic theory that could handle knots in thickened surfaces and would generalize knot theory to arbitrary oriented, not necessarily planar, Gauss diagrams." Since then the theory has grown to include many generalizations of classical knot invariants as well as to provide another context to study Vassilliev invariants. However, understanding the relationship between virtual knots and the standard geometrical constructs associated to knots in $\mathbb{S}^3$ is much harder. To this end, the first author, along with Manturov, initiated the study of virtual covers of links \cite{chrisman2013fibered}. This theory views a virtual knot as a model of a knot in the complement of a fibered link in $\mathbb{S}^{3}$. Here we obtain explicit relations between virtual knot invariants and classical link invariants via virtual covers. This both provides a geometric interpretation of virtual knot invariants and allows for the tools of virtual knot theory to be employed in the study of classical links.    

Recall that the standard geometric interpretation of the Alexander polynomial $\Delta_K(t)$ of a knot $K$ in $\mathbb{S}^3$ is given in terms of a Seifert surface $\Sigma_K$ of $K$. A Seifert matrix $V$ is formed from the pairwise linking numbers of push-offs of a basis of $H_1(\Sigma_K;\mathbb{Z})$. Then $\Delta_K(t)\doteq \det(tV-V^{\uptau})$. An analogous formula for the multi-variable Alexander polynomial $\Delta_L(x,y)$ of a two-component link $L=J \sqcup K$ was discovered by D. Cooper \cite{cooper}. In this case, the single Seifert surface is replaced with a $2$-complex of Seifert surfaces, $S=\Sigma_J \cup \Sigma_K$. The Milnor $\bar{\mu}$-invariants may also be interpreted geometrically using Seifert surfaces. For example, Cochran's link derivatives \cite{cochran1990derivatives} can be used to show that if $L=K_1 \sqcup K_2 \sqcup K_3$ is a 3-component link with vanishing pairwise linking numbers, then $\bar{\mu}_{123}(L)=-\text{lk}(\Sigma_{K_1} \cap \Sigma_{K_2},K_3)$. 

Seifert surfaces, however, are not defined for all virtual knots. Every virtual knot can be represented by a knot in some thickened oriented surface $\Sigma \times [0,1]$. Since not all such knots are homologically trivial, they do not all bound a surface in $\Sigma \times [0,1]$. Virtual covers provide an alternative model in which this obstruction is subverted. Consider a two-component link $L=J \sqcup K$ with $J$ fibered and $\text{lk}(J,K)=0$. The infinite cyclic cover of $J$ is a thickened surface $\Sigma_J \times \mathbb{R}$ for some fiber $\Sigma_J$ of $J$. Thus, $K$ lifts to a knot $\mathfrak{k}$ in $\Sigma_J \times \mathbb{R}$. The knot $\mathfrak{k}$ projects to a virtual knot $\upsilon$. Then $L$, $\mathfrak{k}$ and the covering map $\Sigma_J \times \mathbb{R} \to \mathbb{S}^3\smallsetminus J$ form a \emph{virtual cover} of $L$. The virtual knot $\upsilon$ is called the \emph{associated virtual knot} to $L$. The assignment $L \to \upsilon$ is surjective \cite{chrisman2013fibered}, so that all virtual knots can obtained from some link in $\mathbb{S}^3$. 

It is thus natural to ask if invariants of virtual knots can be expressed in terms of Seifert surfaces of links $L$ in $\mathbb{S}^3$. Here we present two such geometric realizations of virtual knot invariants. The first result considers boundary links $L=J \sqcup K$ where $K$ bounds a Seifert surface $\Sigma_K$ disjoint from a fiber $\Sigma_J$ of $J$. We will show that the associated virtual knot in this case is \emph{almost classical} (i.e. is homologically trivial in some thickened surface representation). Boden et-al. \cite{boden2015virtual} defined an Alexander polynomial for almost classical knots. Here we will denote this by $\overline{\Delta}_{\upsilon}(t)$. The main result of this paper is that $\overline{\Delta}_{\upsilon}(t)$ can be realized as a specialization of the multi-variable Alexander polynomial (MVAP) of the boundary link $L$. More exactly, we show that if $\nabla_{\Sigma_J,\Sigma_K}(t_1,t_2)$ is the MVAP computed from a fiber $\Sigma_J$ of $J$ and a Seifert surface $\Sigma_K$ of $K$ disjoint from $\Sigma_J$, then $\overline{\Delta}_{\upsilon}(t)=\pm t^{2g_K} \cdot \nabla_{\Sigma_J,\Sigma_K}(0,t^{-1})$, where $g_K$ is the genus of $\Sigma_K$.

The second realization result considers the index of a crossing in a virtual knot diagram. For an oriented virtual knot diagram represented by a knot diagram $K$ on a surface $\Sigma$, the index of a crossing $x$ is (up to sign) the algebraic intersection number of the two curves $K_1$, $K_2$ obtained by performing the oriented smoothing at $x$. The index features prominently in the computation of many virtual knot invariants, such as the Henrich-Turaev polynomial \cite{henrich2010sequence} and the writhe polynomial of Cheng \cite{cheng}. We prove that $\upsilon$ can be modeled by a link $L=J \sqcup K$, where $J$ is a connect sum of trefoils or figure-eight knots, and that the index of a classical crossing $x$ of $\upsilon$ is $\bar{\mu}_{123}(J \sqcup K_1 \sqcup K_2)$. Again, $K_1 \sqcup K_2$ is the two-component link obtained by performing the oriented smoothing at $x$. Hence, any invariant defined via the index can be geometrically interpreted in terms $\bar{\mu}_{123}$.

Virtual covers thus provide a way to unmask how virtual knot invariants are hiding inside invariants of classical links. Virtual covers can also be used to indicate geometric properties of classical links. For example, they were used in \cite{chrisman2014virtual} to prove that some classical links are non-invertible. If $L=J \sqcup K$ as above is an invertible link, the assignment $L \to \upsilon$ yields a certain symmetry condition on $\upsilon$. Moreover, there are easily computable virtual knot invariants (e.g. the Sawollek polynomial \cite{saw}) that are not invariant under this symmetry. All together, this suggests that virtual knot theory can be used to extract additional geometric content from classical link invariants. It is thus desirable to extend the set of links on which virtual covers are defined to as large a set of links as possible. Here we introduce \emph{fiber stabilization of links}. We prove that after a fiber stabilization, every multi-component link $L=J\sqcup K$ has a virtual cover. The idea is to add an unknotted component $J_0$ to $L$ so that $J_0 \sqcup J$ is fibered and $K$ has algebraic intersection number $0$ with a fiber of $J_0 \sqcup J$. Applications of fiber stabilization to classical link invariants will be considered in future papers.

The organization of this paper is as follows. First we review virtual knots (Section \ref{sec_background}), virtual covers (Section \ref{sec_vc}), and almost classical knots (Section \ref{sec_ac}). Section \ref{sec_alex} establishes the relationship between the Alexander polynomial of an almost classical knot and the multi-variable Alexander polynomial of a boundary link. Section \ref{sec_index} provides the interpretation of the index in terms of Milnor's triple linking number. Section \ref{sec_fiber} shows that every multi-component link has a fiber stabilization that in turn has a virtual cover. Directions for future research are considered in Section \ref{sec_future}.

\subsection{Models of Virtual Knots} \label{sec_background}

A virtual knot or link may be described via several different models. They are: virtual link diagrams, Gauss diagrams, link diagrams on surfaces, and links in thickened surfaces. Each model has an equivalence relation under which they all coincide as virtual links. 

Consider first \emph{virtual knot diagrams} as defined in \cite{kauffman1998virtual, kauffman2012introduction}. A virtual knot diagram is a generic immersion $\upsilon : \mathbb{S}^{1} \rightarrow \mathbb{R}^{2}$ such that each double point is marked as either a classical crossing or a virtual crossing. Virtual crossings are denoted as circled double points while classical crossing are denoted as usual. Virtual knot diagrams are considered equivalent, denoted by $\leftrightharpoons$, if they are related by a sequence of extended Reidemeister moves (see Figure \ref{fig_rmoves}). Each of the moves $v\Omega 1-v\Omega 4$ may be replaced by the \emph{detour move}, which allows for the erasing of any arc between classical crossings and reconnecting the ends, so long as all new double points are marked as virtual crossings. This is depicted schematically in Figure \ref{fig_rmoves}, bottom left. 

\begin{figure}[htb]
\begin{tabular}{|cc|c|} \hline
 & & \\
\begin{tabular}{ccc}  \def\svgwidth{.35in} 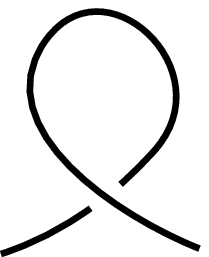 & \begin{tabular}{c} $\leftrightharpoons$ \\ \\ \end{tabular} & 
\def\svgwidth{.35in} 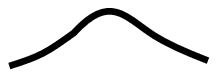 \end{tabular} & \begin{tabular}{ccc} \def\svgwidth{.352in} 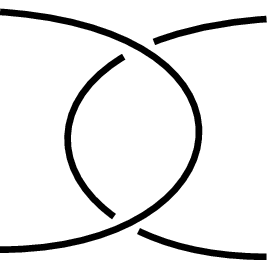 & \begin{tabular}{c} $\leftrightharpoons$ \\ \\ \end{tabular} & \def\svgwidth{.352in} 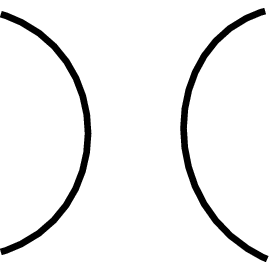 \end{tabular} & \begin{tabular}{ccc}  \def\svgwidth{.352in} 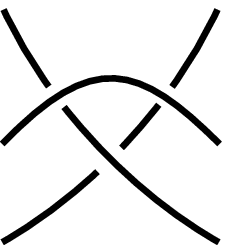 & \begin{tabular}{c} $\leftrightharpoons$ \\ \\ \end{tabular} & 
\def\svgwidth{.352in} 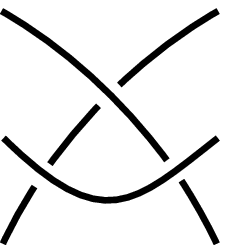 \end{tabular} \\   $\Omega 1$ & $ \Omega 2$ & $\Omega 3$ \\ 
 & & \\
\begin{tabular}{ccc} \def\svgwidth{.35in} 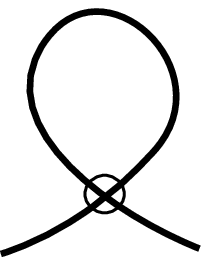 & \begin{tabular}{c} $\leftrightharpoons$ \\ \\ \end{tabular} &  \def\svgwidth{.35in} \input{r1_R.eps_tex} \end{tabular}  & \begin{tabular}{ccc}  \def\svgwidth{.352in} 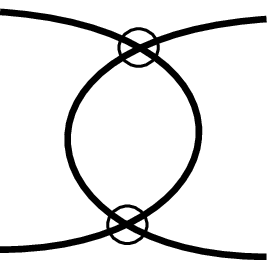 & \begin{tabular}{c} $\leftrightharpoons$ \\ \\ \end{tabular} & 
\def\svgwidth{.352in} \input{r2_R.eps_tex} \end{tabular} & \begin{tabular}{ccc}  \def\svgwidth{.352in} 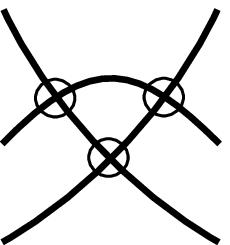 & \begin{tabular}{c} $\leftrightharpoons$ \end{tabular} & 
\def\svgwidth{.352in} 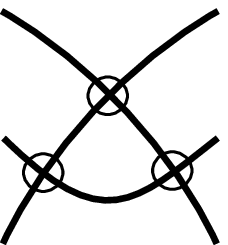 \end{tabular}    \\  
$v\Omega 1$ & $ v\Omega 2$ & $v\Omega 3$ \\ \cline{1-2}
 & & \\ 
\multicolumn{2}{|c|}{\begin{tabular}{c}  \def\svgwidth{.85in} 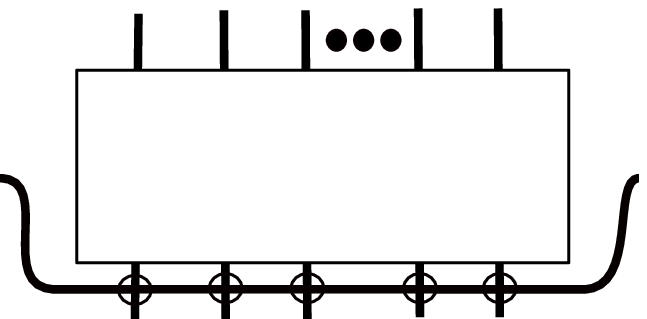 \end{tabular} \begin{tabular}{c} $\leftrightharpoons$ \end{tabular} \begin{tabular}{c}  \def\svgwidth{.85in} 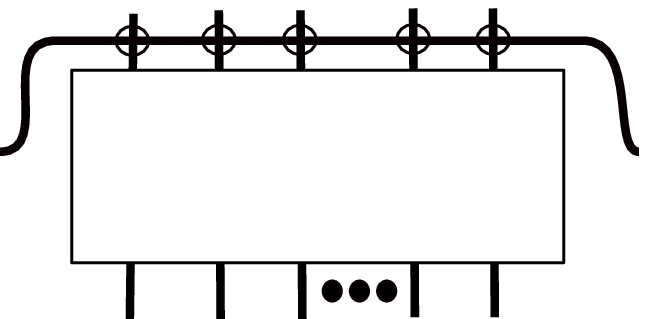 \end{tabular}}  & \begin{tabular}{ccc}  \def\svgwidth{.35in} 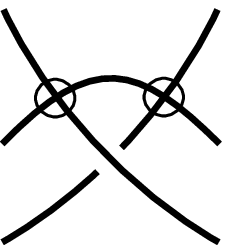 & \begin{tabular}{c} $\leftrightharpoons$ \end{tabular} & \def\svgwidth{.35in} 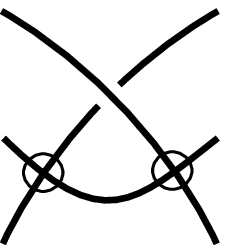 \end{tabular} \\  \multicolumn{2}{|c|}{\underline{The Detour Move}} &  $v\Omega 4$  \\ \hline
\end{tabular} \caption{The extended Reidemeister moves and the detour move.}\label{fig_rmoves}
\end{figure}

A \emph{Gauss diagram} of a virtual knot given by such an immersion $\upsilon:\mathbb{S}^1 \to \mathbb{R}^2$ is a decoration of the domain $\mathbb{S}^{1}$ of $\upsilon$ such that the pre-images of the classical double points are connected by signed arrows that point from the over crossing arc to the under crossing arc. The sign of the arrow corresponds to the sign of a crossing via the standard right hand rule. The equivalence classes of virtual knots correspond to Gauss diagrams modulo orientation preserving diffeomorphisms of $\mathbb{S}^1$ and diagrammatic versions of the classical moves $\Omega 1,\Omega 2, \Omega 3$.

Given a virtual knot diagram $\upsilon$, one may construct a knot diagram on a surface as follows. A disc is placed around each classical crossing. The discs are glued together using untwisted bands that follow the arcs of the diagram. At a virtual crossing, the bands pass over one another. See Figure \ref{fig_four_models} (3). If discs are attached to the boundary components of this surface, we obtain the \emph{Carter surface} of $\upsilon$ \cite{carter1991classifying}. 

Conversely, let $\Sigma$ be a compact connected oriented (cco) smooth surface and $\mathfrak{k}$ a knot in $\Sigma \times [0,1]$. Such knots will be considered equivalent up to ambient isotopy, orientation preserving diffeomorphisms of $\Sigma$, and \emph{stablization}/\emph{destabilzation}. Stabilization/destabilization is the relation defined by removing/adding 1-handles from $\Sigma$ that do not intersect a knot diagram of $\mathfrak{k}$ projected to $\Sigma$.  Equivalence classes of knots in thickened surfaces are then in one-to-one correspondence with virtual knots. Kuperberg \cite{kuperberg2003virtual} furthermore showed that there is a knot in a thickened surface of least genus corresponding to each virtual knot. 

\begin{figure}[htb]
\fcolorbox{black}{white}{
$\begin{array}{cc} \\
\begin{array}{c}
\def\svgwidth{1in}
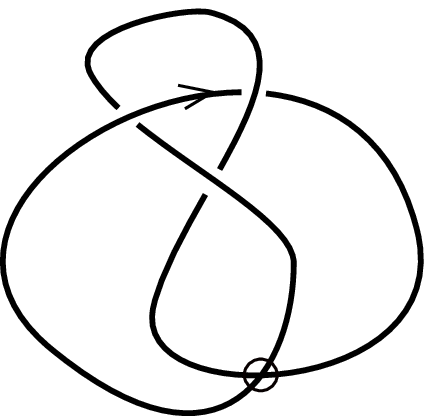 \\ (1) \end{array} & \begin{array}{c}
\def\svgwidth{1in}
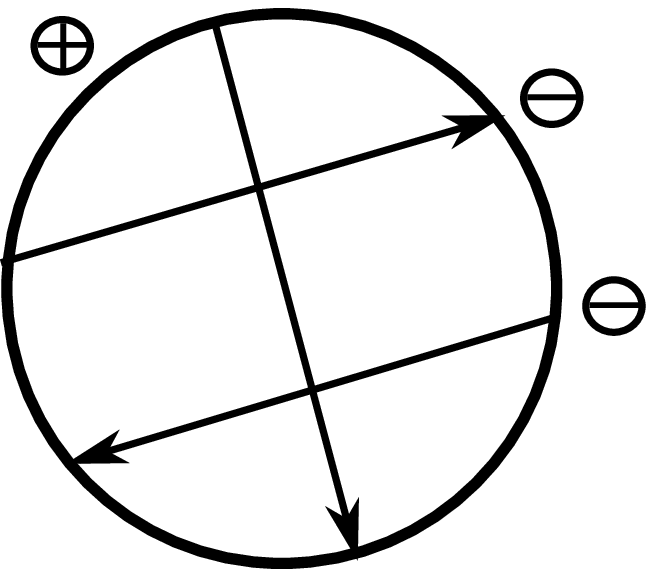 \\ (2) \end{array} \\  \\ \begin{array}{c}
\def\svgwidth{1in}
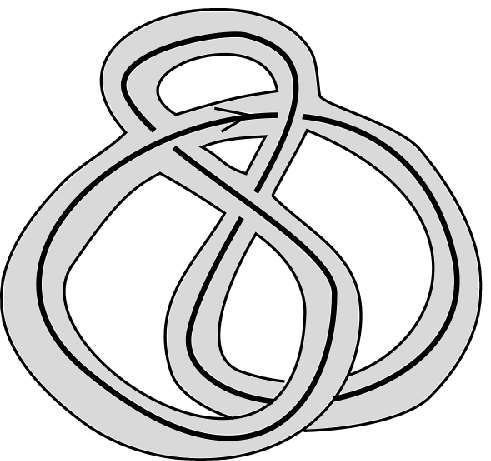 \\ (3) \end{array} & \begin{array}{c}
\def\svgwidth{1.72in}
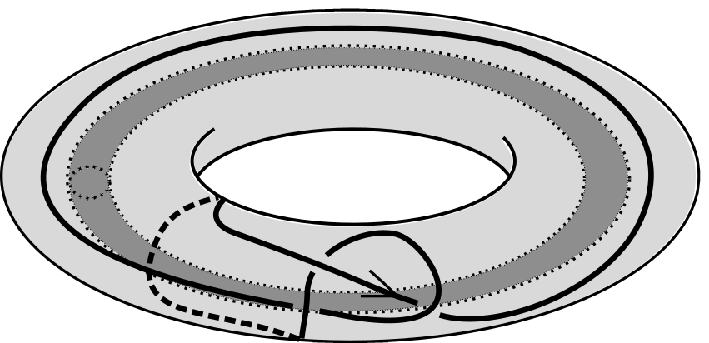 \\ (4) \end{array} \\ 
\end{array}
$}
\caption{The four equivalent models: (1) virtual knot diagrams, (2) Gauss diagrams, (3) knot diagrams on surfaces, and (4) knots in thickened surfaces.}\label{fig_four_models}
\end{figure}

\subsection{Virtual Covers of Links} \label{sec_vc} Throughout the text, we will denote by $X\smallsetminus \nu Y$ the closed space obtained by deleting an open tubular neighborhood of $Y$ from $X$. For a knot $J$, let $N_J=\mathbb{S}^3\smallsetminus \nu J$ denote the knot exterior. Recall that a knot $J$ is fibered if it has a Seifert surface $\Sigma_J$ such that the pairs $(\mathbb{S}^3\smallsetminus \nu \Sigma_J,(\mathbb{S}^3\smallsetminus \nu\Sigma_J) \cap \partial N_J)$ and $(\Sigma_J \cap N_J,\partial (\Sigma_J \cap N_J)) \times \mathbb{I}$ are diffeomorphic. In other words, cutting out $\Sigma_J$ produces a thickened surface. The surface $\Sigma_J$ is called a \emph{fiber}. It is a minimal genus Seifert surface for $J$ (see \cite{kawauchi}, Theorem 4.1.10). By Stallings theorem \cite{stall_certain}, a knot is fibered if and only if the commutator subgroup of the knot group $\pi_1(N_J,z_0)$ is finitely generated and free. Thus $N_J$ admits a covering space $\Pi_J:(\Sigma_J \cap N_J) \times \mathbb{R} \to N_J$, where $(\Pi_J)_*(\pi_1((\Sigma_J \cap N_J) \times \mathbb{R},x_0)) \cong [\pi_1(N_J,z_0),\pi_1(N_J,z_0)]$.

For a knot $K$ in a manifold $N$ we write $K^N$. Let $\Sigma$ be a cco smooth surface. Suppose that $N$ is any cco $3$-manifold admitting a regular orientation preserving covering space $\Pi:\Sigma \times \mathbb{R} \to N$. Suppose that $K^N$ is a knot in $N$ and that there is a knot $\mathfrak{k}^{\Sigma \times \mathbb{R}}$ such that $\Pi(\mathfrak{k})=K$. The knot $\mathfrak{k}$ stabilizes to a virtual knot $\upsilon$. The triple $(\mathfrak{k}^{\Sigma \times \mathbb{R}}, \Pi,K^N)$ is called a \emph{virtual cover} of $K$ and $\upsilon$ is called the \emph{associated virtual knot}. Specific details on virtual covers including basic properties and invariance are addressed in \cite{chrisman2013fibered}.

\begin{figure}[htb]
\fcolorbox{black}{white}{
$\begin{array}{cc} \\
\def\svgwidth{3in}
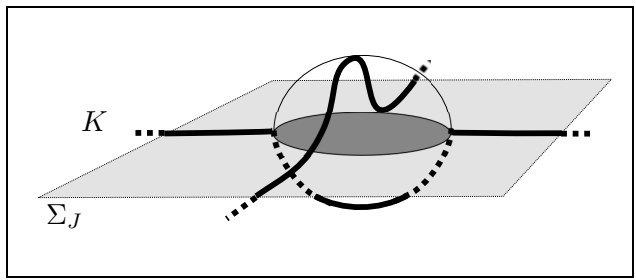 \\ \\
\end{array}
$}
\caption{A crossing in a ball.}\label{fig_cross}
\end{figure}

Virtual covers for links in $\mathbb{S}^3$ arise from the following construction. Let $L=J \sqcup K$ be a two component link with $J$ fibered and $\text{lk}(J,K)=0$. Then there is a fiber $\Sigma_J$ of $J$, a covering space $\Pi:\Sigma_J \times \mathbb{R} \to N_J$, and a knot $\mathfrak{k}^{\Sigma_J \times \mathbb{R}}$ such that $\Pi(\mathfrak{k})=K$. The knot $\mathfrak{k}$ projects to a virtual knot $\upsilon$. If the link $L$ is in \emph{special Seifert form} (SSF), then $\upsilon$ can be computed directly from a link diagram.  We next sketch the definition of SSF and the computation of $\upsilon$. For a precise definition, see \cite{chrisman2014virtual}.

Figure \ref{fig_cross} shows a configuration of arcs called a \emph{crossing in a ball}. It consists of a $3$-ball $B$ embedded in a coordinate neighborhood of a point on $\Sigma_J$ in a tubular neighborhood $\nu\Sigma_J$ such that $B \cap \Sigma_J$ in a disc. The disc divides $B$ into an upper and lower hemisphere, where the over-crossing arc lies in the upper hemisphere and the under-crossing one lies in the lower hemisphere. Suppose a pairwise disjoint collection $B_1,\ldots,B_p$ of crossing in balls on $\Sigma_J$ are joined together at their arc endpoints by a collection of disjoint simple arcs in $\overline{\Sigma_J\smallsetminus \sqcup_{i=1}^p B_i}$ so that the result is a knot $K\subset \mathbb{S}^3$. Then we say that $K$ is in \textit{special Seifert form} (SSF). The associated virtual knot, then, is the virtual knot corresponding to the diagram of $K$ on $\Sigma_J$ (see Figure \ref{fig_virt_cov}). In particular, there is a one-to-one correspondence between classical crossings of the associated virtual knot and self-crossings of $K$ as a diagram on $\Sigma_{J}$. SSF also extends to $n+1$ component links $L=J \sqcup K$ with $J$ a fibered $n$ component link and $K$ a knot. The following theorem, proved in \cite{chrisman2015virtual}, shows that the associated virtual knot functions as a invariant of SSF links. The links in the statement are \emph{ordered} (or \emph{colored}) in the sense that the $n$ components of $L$ have a fixed labeling as $1,\ldots, n$. Two ordered oriented links $L_1,L_2$ are equivalent, denoted $L_1 \leftrightharpoons L_2$ if there is an ambient isotopy taking $L_1$ to $L_2$ that preserves both the orientations and the ordering.

\begin{figure}[htb]
\fcolorbox{black}{white}{
$\begin{array}{ccc} \\
\begin{array}{c}
\def\svgwidth{1.45in}
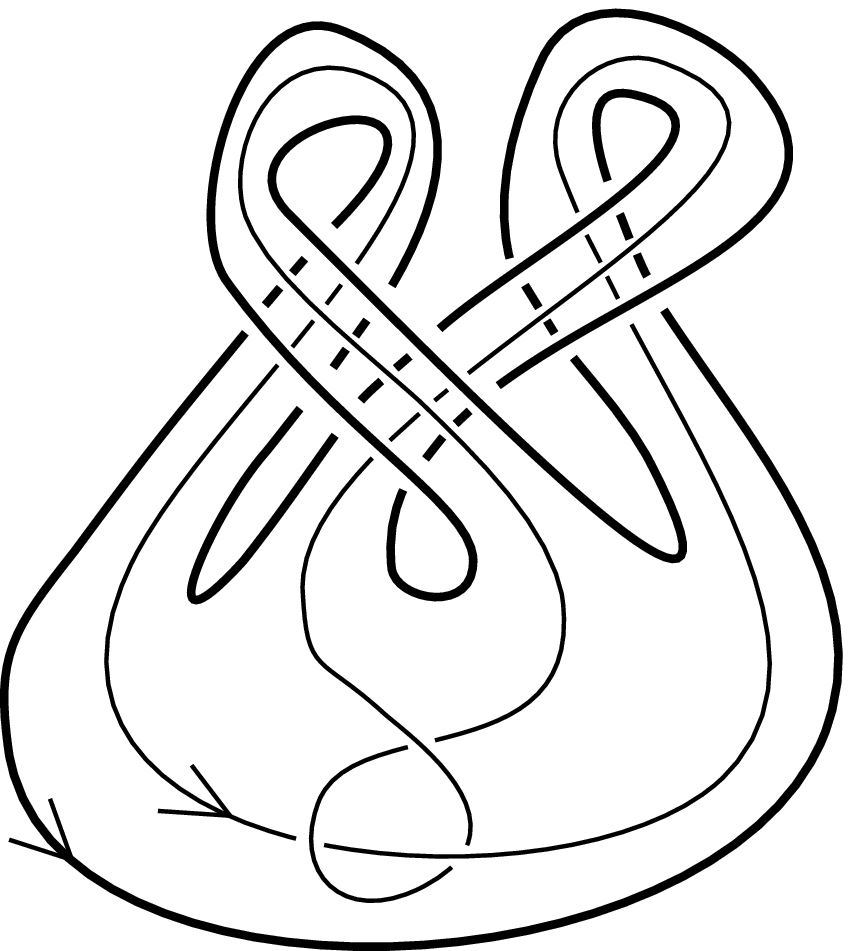 \end{array} & \to &\begin{array}{c}
\def\svgwidth{1.25in}
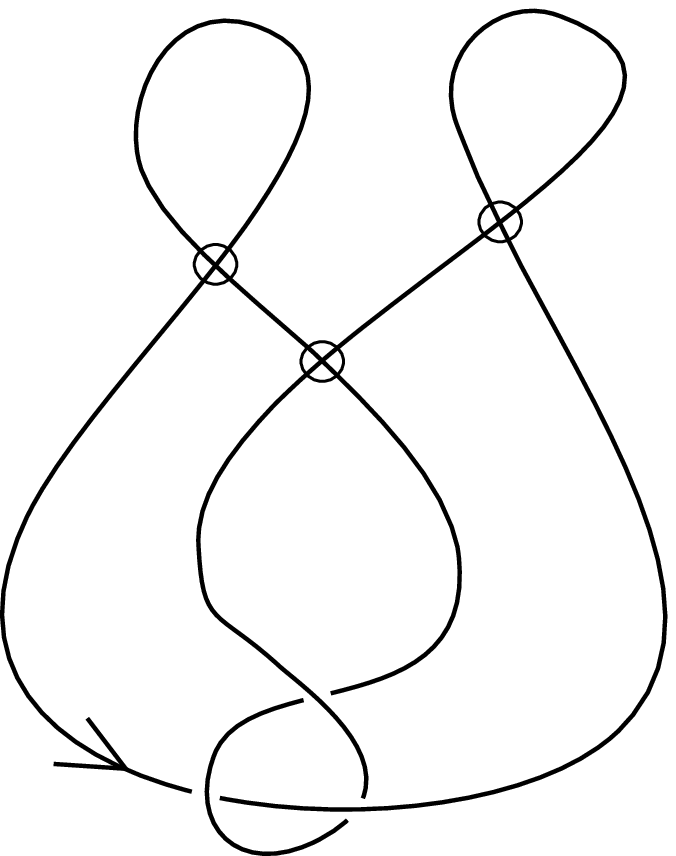\end{array}  \\ 
\end{array}
$}
\caption{(Left) A link $L=J \sqcup K$ in SSF with $J\leftrightharpoons 4_1$. An evident fiber of $J$ is depicted in disc-band form. (Right) The associated virtual knot.}\label{fig_virt_cov}
\end{figure}

\begin{theorem} \cite{chrisman2015virtual} Let $L_1=J_1 \sqcup K_1$ and $L_2=J_2 \sqcup K_2$ be $n+1$ component links in SSF, where $J_1,J_2$ are $n$-component fibered links. Then the associated virtual knots $\upsilon_1$ and $\upsilon_2$ for $L_1$ and $L_2$, respectively, are invariant. Moreover, if $L_1 \leftrightharpoons L_2$ as links, then $\upsilon_1 \leftrightharpoons \upsilon_2$ as virtual knots.
\end{theorem}

For geometric applications of virtual covers to links, knots in $3$-manifolds, and link concordance, the reader should consult \cite{chrisman2014virtual,chrisman2015virtual, chrisman2013fibered}.

\subsection{Almost Classical Knots} \label{sec_ac} A virtual knot diagram $\upsilon$ is said to be \emph{Alexander numerable} if the arcs between the classical crossings may be labeled by integers $\lambda_1,\cdots,\lambda_n$ so that the consistency equations in Figure \ref{fig_ac_defn} are satisfied at each classical crossing. A virtual knot that admits an Alexander numerable diagram is said to be an \emph{almost classical knot}. Every classical knot is almost classical but not every almost classical knot is classical. Almost classical knots were first introduced as an object of independent study by Silver-Williams \cite{silwill_AC}. An equivalent condition for a virtual knot diagram to be Alexander numerable is that the index every crossing is zero (see below, Section \ref{sec_milnor}). The equivalence of the two definitions follows, for example, from the Cairns-Elton criterion \cite{cairnselton}. 
\newline

\begin{figure}[htb]
\fcolorbox{black}{white}{
$\begin{array}{cc} \\
\begin{array}{c}
\def\svgwidth{1.2in}
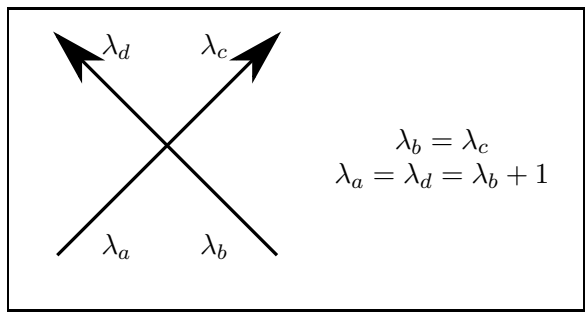 \end{array} & \begin{array}{c} \lambda_b=\lambda_c \\ \lambda_a=\lambda_d=\lambda_b+1\end{array} \\ \\
\end{array}
$}
\caption{Alexander numbering of a knot diagram.}\label{fig_ac_defn}
\end{figure}

Almost classical knots are the virtual knots for which the concept of Seifert surfaces makes sense. Indeed, Silver-Williams observed that certain arguments in classical knot theory that utilize Seifert surfaces can be extended to virtual knots when they have an Alexander numbering. Boden-et al.\cite{boden2015alexander} showed that a knot $\mathfrak{k}$ in a thickened surface $\Sigma \times [0,1]$ bounds an oriented compact surface $F$ if and only if it has a diagram on $\Sigma$ that is Alexander numerable. Such a knot projects to a virtual knot that is almost classical. 

Furthermore, Boden-et al.\cite{boden2015alexander} gave an algorithm for constructing a spanning surface $F$ for any homologically trivial knot in a thickened surface $\Sigma \times [0,1]$. The algorithm mirrors the Seifert surface algorithm for classical knots. First one performs the oriented smoothing at each crossing of the knot diagram on $\Sigma$. The result is a set $\{\gamma_1,\ldots,\gamma_n\}$ of simple oriented closed curves on $\Sigma$. Since $\gamma_1 \cup \cdots \cup \gamma_n$ is homologically trivial, there is a collection of oriented connected subsurfaces $S_1,\ldots,S_m$ of $\Sigma$ such that $\partial S_j \ne \emptyset$ for $1 \le j \le m$ and $\bigcup_{j=1}^m \partial S_j =\bigcup_{i=1}^n \gamma_i$. The spanning surface $F$ is obtained from $S_1,\ldots,S_m$ by placing overlapping subsurfaces at different heights and gluing in half-twisted bands at the smoothed crossings of $\mathfrak{k}$. This last step is the familiar Seifert surface algorithm. It is important to note that in the virtual case the subsurfaces $S_1,\ldots,S_m$ may have any genus and any number ($\ge 1$) of boundary components. The algorithm for classical knots, on the other hand, uses only discs. 

Recall that the Alexander polynomial of a classical knot $K$ can be computed from a Seifert surface $\Sigma_K$ of genus $g_K$. Let $a_1,\cdots, a_{2g_K}$ be a collection of simple closed curves on $\Sigma_K$ representing a basis for $H_1(\Sigma_K;\mathbb{Z})$. Such a collection of simple closed curves is often called a \emph{canonical system}. Let $V=(\text{lk}(a_i^{-},a_j))$ be the $2g_K \times 2g_K$ matrix of linking numbers, where $a^{\pm}$ denotes the $\pm$ push-off of $a$ into $\mathbb{S}^3\smallsetminus\nu \Sigma_K$. Then the Alexander polynomial is given by $\Delta_K(t)=\det(tV-V^\uptau)$, which is well defined up to a multiple of powers of $t^{\pm 1}$. For all versions of the Alexander polynomial used in this paper, this indeterminacy is denoted by $\doteq$. 

Since almost classical knots have a Seifert surface in some thickened surface, it also possible to define an Alexander polynomial for almost classical knots. To do this, it is first necessary to have a definition of the linking number in $\Sigma \times [0,1]$. Let $\mathfrak{k}_1,\mathfrak{k}_2$ be knots in $\Sigma \times [0,1]$. Then $H_1(\Sigma \times [0,1]\smallsetminus \nu (\mathfrak{k}_1),\Sigma \times 1)$ is infinite cyclic and generated by a meridian $\mu$ of $\mathfrak{k}_1$ (cf. \cite{boden2015virtual}, Proposition 7.1). Then $[\mathfrak{k}_2]=\alpha \cdot \mu$ for some $\alpha \in \mathbb{Z}$. Define the \emph{linking number of} $\mathfrak{k}_1$ \emph{and} $\mathfrak{k}_2$ \emph{in} $\Sigma \times [0,1]$ to be $\text{lk}_{\Sigma}(\mathfrak{k}_1,\mathfrak{k}_2)=\alpha$. It is important to note that the $\text{lk}_{\Sigma}$ is not symmetric. However, we have the following relation (see Cimasoni-Turaev \cite{ct}, Section 1.2):
\[
\text{lk}_{\Sigma}(\mathfrak{k}_1,\mathfrak{k}_2)-\text{lk}_{\Sigma}(\mathfrak{k}_2,\mathfrak{k}_1)=p_*([\mathfrak{k}_1]) \cdot p_*([\mathfrak{k}_2]),
\]
where $p:\Sigma \times [0,1] \to \Sigma$ is projection onto the first factor and $\cdot$ represents the intersection form on $\Sigma$. The consequence of this asymmetry is that two Seifert matrices are needed to define the Alexander polynomial of almost classical knots. Let $\Sigma_{\mathfrak{k}}$ be a Seifert surface of genus $g_{\mathfrak{k}}$ for $\mathfrak{k}$ in $\Sigma \times [0,1]$ and let $a_1,\ldots,a_{2g_{\mathfrak{k}}}$ be a canonical system of curves on $\Sigma_{\mathfrak{k}}$. Then the $\pm$-Seifert matrices $V^{\pm}$ are given by the $2g_{\mathfrak{k}} \times 2g_{\mathfrak{k}}$ matrices $V^{\pm}=(\text{lk}_{\Sigma}(a_i^{\pm},a_j))$.

\begin{definition}[Alexander polynomial of almost classical knots \cite{boden2015virtual}] Let $\upsilon$ be an almost classical knot. Let $\mathfrak{k}$ be a homologically trivial knot in $\Sigma \times [0,1]$ representing $\upsilon$. Let $\Sigma_{\mathfrak{k}}$ be a Seifert surface for $\mathfrak{k}$, $g_{\mathfrak{k}}$ the genus of $\Sigma_{\mathfrak{k}}$, and $V^{\pm}$ the $\pm$-Seifert matrices relative to some set of simple closed curves $\{a_1,\ldots,a_{2g_{\mathfrak{k}}}\}$ representing a basis of $H_1(\Sigma_{\mathfrak{k}};\mathbb{Z})$. Define a polynomial: 
\[
\overline{\Delta}_{\Sigma_{\mathfrak{k}}}(t)=\det(tV^--V^+).
\]
The \emph{Alexander polynomial} of the almost classical knot $\upsilon$, denoted $\overline{\Delta}_{\upsilon}(t)$, is the element of $\mathbb{Z}[t,t^{-1}]/\left<\pm t^k-1:k \in \mathbb{Z}\right>$ defined by $\overline{\Delta}_{\upsilon}(t)\doteq \overline{\Delta}_{\Sigma_{\mathfrak{k}}}(t)$.
\newline 
\end{definition}

By \cite{boden2015virtual}, Section 7, $\Delta_{\upsilon}(t)$ is a well-defined invariant of almost classical knots that is independent of the choice of $\Sigma$, $\mathfrak{k}$, and $\Sigma_{\mathfrak{k}}$. For any classical knot $K$, $\overline{\Delta}_K(t)=\Delta_K(t)$. As in the classical case, $\overline{\Delta}_{\upsilon}(t)$ is a generator of the first elementary ideal of the Alexander module (see \cite{boden2015alexander}, Corollary 7.3). The polynomial $\overline{\Delta}_{\Sigma_{\mathfrak{k}}}(t)$ is not an invariant of $\upsilon$ itself, but it is independent of the choice of basis $\{a_1,\ldots,a_{2g_{\mathfrak{k}}}\}$. Indeed, a change to another basis gives matrices $M^{\uptau} V^+ M$ and $M^{\uptau}V^-M$, where $M$ is an integral unimodular matrix.

\subsection{Defining the MVAP}

The multi-variable Alexander polynomial of a two component link in $\mathbb{S}^3$ (abbreiviated as MVAP) is the natural extension of the Alexander polynomial of a knot. That is, the multi-variable Alexander polynomial derives from a pair of Seifert forms on the universal abelian cover of of the link complement \cite{cooper1982signatures}. For an arbitrary link, this can be computed from a $2$-complex made from a union of Seifert surfaces of the components. 

In the case of a boundary link the computation is particularly simple. Recall that a link $J \sqcup K$ is a boundary link if there are Seifert surfaces $\Sigma_J$ of $J$ and $\Sigma_K$ of $K$ such that $\Sigma_J \cap \Sigma_K=\emptyset$. The Seifert forms can then be computed using linking numbers of canonical systems of curves on the disjoint Seifert surfaces. At the level of the first elementary ideal of the Alexander module, the MVAP of a boundary link vanishes. However, Gutierrez \cite{gutierrez} showed that the Alexander polynomial associated to the Seifert form of a boundary link is a generator of the second elementary ideal of the Alexander module. These are not always vanishing for boundary links and hence serve as useful invariants. To make this distinction clear, we will henceforth use the notation $\nabla_{J,K}$, rather than $\Delta_L$, to denote the MVAP corresponding to a generator of the second elementary ideal.

Friedl \cite{friedl2006algorithm} gave a succinct description of $\nabla_{J,K}$ for boundary links. Consider a boundary link $J \sqcup K$ and suppose we have Seifert surfaces $\Sigma_J,\Sigma_K$, such that $J=\partial \Sigma_J$, $K=\partial \Sigma_K$ and $\Sigma_J \cap \Sigma_K=\emptyset$. If the genera of $\Sigma_J,\Sigma_k$ are $g_1,g_2$, respectively, then we may form a symplectic basis for $H_1(\Sigma_J \sqcup \Sigma_K;\mathbb{Z})$ from simple closed curves $a_{1},\ldots,a_{2g_{1}}$ on $\Sigma_J$ and $a_{2g_1+1},\ldots, a_{2(g_1+g_2)}$ on $\Sigma_K$. Let $A$ be the Seifert matrix of $J \sqcup K$, so that  $(i,j)$-entry is:
\[
A_{i,j}=\text{lk}(a_{i}^{-},a_{j}).
\]
Notice that letting $A_{J}$ and $A_{K}$ be Seifert matrices for each of the link components individually we see that $A$ may be written in block form as:
\begin{equation} \label{eqn_seif_mat}
A=\left[\begin{array}{c|c} A_J & B \\ \hline B^\uptau & A_K\end{array}\right]
\end{equation}
Recall that $H_{1}(\mathbb{S}^3\smallsetminus \nu (J \sqcup K);\mathbb{Z})\cong \mathbb{Z}\oplus \mathbb{Z}$ and is generated by meridians $t_1$ of $J$ and $t_2$ of $K$. Let $\widetilde{X}$ denote the universal abelian cover of the link complement. Then $H_1(\widetilde{X})$ is a finitely generated $\mathbb{Z}[\mathbb{Z}\oplus \mathbb{Z}]$-module, where we identify $\mathbb{Z}[\mathbb{Z}\oplus\mathbb{Z}]$ with $\Lambda_2=\mathbb{Z}[t_1^{\pm 1},t_2^{\pm 1}]$. Now, let $T$ be the matrix whose $(i,j)$-entry is:
\begin{equation*}
T_{i,j}=
\begin{cases}
t_{1} &  i=j\ \text{and}\ 1 \leq i\leq 2g_{1}\\
t_{2} &  i=j\ \text{and}\ 2g_{1} +1 \leq i \leq 2(g_{1}+g_{2})\\
0 & \text{otherwise}.
\end{cases}
\end{equation*}
Gutierrez proved that $H_1(\widetilde{X})\cong \Lambda_2 \oplus \frac{\Lambda_2^{m}}{(AT-A^{\uptau})\Lambda_2^m}$, where $m=2(g_1+g_2)$ (\cite{gutierrez}, Corollary 3). Then $\det(AT-A^{\uptau})$ provides the generator $\nabla_{J,K}$ discussed above. We record this in the following definition. 

\begin{definition}[MVAP of a boundary link] Let $L=J \sqcup K$ be a two component boundary link and $\Sigma_J, \Sigma_K$ disjoint Seifert surfaces of $J,K$, respectively. Define a polynomial $\nabla_{\Sigma_J,\Sigma_K}(t_1,t_2) \in \mathbb{Z}[t_1,t_2]$ by:
\[
\nabla_{\Sigma_J,\Sigma_K}(t_{1},t_{2})=\det(AT-A^{\uptau}),
\]
where $A$ and $T$ are as defined above. The \emph{multi-variable Alexander polynomial} (MVAP) of $J \sqcup K$, denoted $\nabla_{J,K}(t_1,t_2)$ is defined by $\nabla_{J,K}(t_1,t_2)\doteq \nabla_{\Sigma_J,\Sigma_K}(t_1,t_2)$, where the polynomial is well-defined up to multiplication by units of $\Lambda_2=\mathbb{Z}[t_1^{\pm 1}, t_2^{\pm 1}]$.  The MVAP is an invariant of boundary links and is independent of the choice of $\Sigma_J,\Sigma_K$ (see \cite{friedl2006algorithm}, Proposition 1.1).
\end{definition}

\section{Relating virtual Alexander polynomials to the MVAP} \label{sec_alex}

\subsection{Theorem statement} Our main goal is to relate invariants of the associated virtual knot of a virtual cover to standard invariants of the link to which it corresponds. To that end we offer the following theorem. 

\begin{theorem}\label{thm_main_alex} Let $J$ be a fibered knot. Suppose $L=J \sqcup K$ is a two component boundary link such that $K$ bounds a Seifert surface $\Sigma_K$ disjoint from a fiber $\Sigma_J$ of $J$. Let $g_K$ the genus of $\Sigma_K$. Then the associated virtual knot $\upsilon$ to $L$ is almost classical and we have:
\begin{eqnarray}
\overline{\Delta}_{\upsilon}(t)  &\doteq & \overline{\Delta}_{\Sigma_K}(t) = \pm t^{2g_K} \cdot \nabla_{\Sigma_J,\Sigma_K}(0,t^{-1}), \\
\Delta_J(t) &\doteq & \nabla_{J,K}(t,1), \text{ and} \\ 
\Delta_K(t) & \doteq &\nabla_{J,K}(1,t)
\end{eqnarray}
Moreover, the $\pm$ sign in (2) is determined by $\det(A_J)=\pm 1$, where $A_J$ is a Seifert matrix for the Seifert surface $\Sigma_J$.
\end{theorem}

\begin{remark}  Relations (3) and (4) are easy consequences of the definition and are stated in the theorem simply for the purposes of comparison. They are similar in form to the well-known Torres conditions (e.g. see \cite{kawauchi}, Theorem 7.4.1).
\end{remark}

\subsection{Supporting lemmas} The key idea in the proof of Theorem \ref{thm_main_alex} is to relate linking numbers in $\mathbb{S}^{3}$ to linking numbers in $\Sigma_{J} \times I$, where $\Sigma_J$ is a fiber of $J$. This will allow us to compare the respective Seifert forms. This is accomplished in the following lemma. Recall that since $J$ is fibered, $\mathbb{S}^3\smallsetminus \nu\Sigma_J$ is diffeomorphic to $\Sigma_J \times [0,1]$. Removing a tubular neighborhood of $\Sigma_J$ leaves two copies of $\Sigma_J$, denoted $\Sigma_J^{\pm}$, corresponding to the $\pm$ push-offs of $\Sigma_J$, respectively. We will identify $\Sigma_J^{+}$ with $\Sigma_J\times\{1\}$ and $\Sigma_J^-$ with $\Sigma_J \times \{0\}$. Thus, $\Sigma_J^+$ is the ``top'' of the thickened surface $\Sigma_J \times [0,1]$.

\begin{lemma} \label{lemma_linking}
Let $J$ be a fibered knot and $\Sigma_{J}$ a fiber for $J$ of genus $g$. Suppose that 
$y$ is a knot in $\mathbb{S}^{3}\smallsetminus \nu \Sigma_{J}$ and that $x$ is a knot in $(\mathbb{S}^{3}\smallsetminus\nu\Sigma_{J})\smallsetminus \nu (y)$. Let $\beta=\{a_{1}, \ldots, a_{2g}\}$ be a basis for the first homology of $\Sigma_{J}$. For an arbitrary knot $z$ in $\mathbb{S}^{3}\smallsetminus\nu\Sigma_{J}$ let $l_{z}^{\uptau}=(lk(a_{1},z),\ldots, lk(a_{2g},z)))$. Then:
\[
lk_{\Sigma_{J}}(y,x)=lk(y,x)-l_{y}^{\uptau}A_J^{-1}l_{x},
\]
where $A_J$ is the Seifert matrix for $\Sigma_{J}$ with respect to $\beta$.

\end{lemma}

\begin{proof} We may assume without loss of generality that $\Sigma_J$ is in disc-band form. Furthermore, assume that the elements $a_i$ of $\beta$ are represented by simple closed curves on $\Sigma_J$, also denoted $a_i$. Each $a_i$ will be assumed to pass along the core of exactly one band. A basis for $H_1(\mathbb{S}^3\smallsetminus\nu\Sigma_J)$ is given by $\beta^*=\{a_1^*,\ldots,a_{2g}^*\}$, where each $a_j^*$ is an unknot encircling the band of $a_j$ and oriented so that $\text{lk}(a_j,a_j^*)=1$. Thus, given any $[z] \in H_1(\mathbb{S}^3\smallsetminus \nu\Sigma_J)$, we may write:
\[
[z]=\sum_{i=1}^{2g} \text{lk}(a_i,z) [a_i^*].
\] 
By a Mayer-Vietoris argument ($\mathbb{Z}$ coefficients everywhere), we have the following decomposition:
\[
H_1(\mathbb{S}^3\smallsetminus\nu(\Sigma_J \sqcup y)) \cong H_1(\Sigma_J^+) \oplus H_1(\mathbb{S}^3\smallsetminus\nu(\Sigma_J \sqcup y),\Sigma_J^+).
\]
After substituting in our conventions, we have the decomposition:
\[
H_1(\Sigma_J \times [0,1]\smallsetminus \nu (y)) \cong H_1(\Sigma_J \times \{1\}) \oplus H_1(\Sigma_J \times [0,1]\smallsetminus \nu(y) ,\Sigma_J \times \{1\}).
\]
As previously discussed in Section \ref{sec_ac}, the second factor is freely generated by a meridian $\mu$ of $y$. Hence, $H_1(\mathbb{S}^3\smallsetminus\nu(\Sigma_J \sqcup y))$ is freely generated by $\{\mu,a_1,\ldots,a_{2g}\}$. More exactly, the summand $H_1(\Sigma_J^+)$ is generated by copies $\{a_1^+,\ldots,a_{2g}^+\}$ on $\Sigma_J^+$, but this notation will be hereafter suppressed.  Now writing $[x]$ in this basis, we have:  
\[
[x]=r_0[\mu]+\sum_{i=1}^{2g} r_i [a_i].
\]
For $[z]\in H_1(\Sigma_J)$, let $[z]_{\beta}$ denote the coordinate vector of $[z]$ in the basis $\beta$. Then $A_J \cdot [z]_{\beta}$ is the element of $H_1(\mathbb{S}^3\smallsetminus\nu\Sigma_J)$ corresponding to the positive push-off of $z$ in the basis $\beta^*$ (see \cite{bz}, Lemma 8.6). Since $J$ is fibered, $A_J$ is invertible over $\mathbb{Z}$ (see \cite{bz}, Proposition 8.6). It follows that $A_J^{-1}[x]_{\beta^*}$ is the homology class of $x$ projected onto $\Sigma_J^+=\Sigma_J \times \{1\}$ in the basis $\beta$.

Now consider the induced map of the inclusion $\mathbb{S}^3\smallsetminus\nu(\Sigma_J \sqcup y) \to \mathbb{S}^3\smallsetminus\nu(J \sqcup y)$ in homology. For any knot $z$ in $\mathbb{S}^3\smallsetminus\nu(\Sigma_J \sqcup y)$, $\text{lk}(J,z)=0$. This implies $[z] \to \text{lk} (y,z) [\mu]$. Hence:
\[
[x] \to \left(r_0+\sum_{i=1}^{2g} r_i \cdot \text{lk}(y,a_i) \right)\cdot [\mu]
\]
Furthermore, our prior observations imply that:
\[
\left[ \begin{array}{c} r_1 \\ \vdots \\ r_{2g} \end{array} \right]_{\beta}= A_J^{-1} \left[ \begin{array}{c} \text{lk}(a_1,x) \\ \vdots \\ \text{lk}(a_{2g},x) \end{array} \right]_{\beta^*}
\]
Lastly, note that by definition of $\text{lk}_{\Sigma_J}$, $r_0=\text{lk}_{\Sigma_J}(y,x)$.
\end{proof}

We return now to the proof of Theorem \ref{thm_main_alex}. Let $J,K$ be as in the statement of Theorem \ref{thm_main_alex}. Suppose that $\Sigma_J$ is a fiber of $J$ and that $\Sigma_K$ is a Seifert surface of $K$ such that $\Sigma_J \cap \Sigma_K=\emptyset$. Then $K$ bounds a Seifert surface in $\mathbb{S}^3\smallsetminus\nu\Sigma_J$, which is identified via a diffeomorphism as above with $\Sigma_J \times [0,1]$. Thus the link $L=J \sqcup K$ has a virtual cover such that the associated virtual knot $\upsilon$ is almost classical. Lemma \ref{lemma_linking} now allows for the computation of the Seifert matrices $V^{\pm}$ for $\upsilon$ in terms of the block decomposition of the Seifert matrix $A$ of $L$ from Equation \ref{eqn_seif_mat}.

\begin{corollary} For a basis of simple closed curves $a_{2g_1+1},\ldots,a_{2(g_1+g_2)}$ of $\Sigma_K$, let  $V^{\pm}=(\text{lk}_{\Sigma_{J}}(a_i^{\pm},a_j))$ be the Seifert matrices of $\upsilon$. Then in the notation of Equation \ref{eqn_seif_mat}, we have that:
\[
V^-=A_K-B^{\uptau} A_J^{-1}B\ \text{and}\ V^+=A_K^{\uptau}-B^{\uptau}A_J^{-1}B.
\]
\end{corollary}
\begin{proof} This follows from Lemma \ref{lemma_linking} and some elementary matrix algebra.  
\end{proof}

\subsection{Proof of Theorem \ref{thm_main_alex}} We compute $\det(AT-A^{\uptau})$ as follows:
\[
A=\left[\begin{array}{c|c} A_J & B \\ \hline B^{\uptau} & A_K\end{array}\right] \implies \det(AT-A^{\uptau})=\left|\begin{array}{c|c} t_1 A_J-A_J^{\uptau} & (t_2-1) B \\ \hline (t_1-1) B^{\uptau} & t_2 A_K-A_K^{\uptau} \end{array}\right|.
\] 
Setting $t_2=1$ immediately gives the second claim and setting $t_1=1$ gives the third claim. Now, from the proceeding lemmas we have that:
$$s V^--V^+=(s A_K-A_K^{\uptau})-(s-1)B^{\uptau} A_J^{-1}B.$$
Next, set $t_1=0$ in $AT-A^{\uptau}$. To get a more recognizable form, multiply on the left by a matrix of determinant 1.
\begin{eqnarray*}
& & \left[\begin{array}{c|c} I & O \\ \hline -B^{\uptau}(A_J^{\uptau})^{-1} & I \end{array}\right] \cdot \left[\begin{array}{c|c} -A_J^{\uptau} & (t_2-1) B \\ \hline -B^{\uptau} & t_2 A_K-A_K^{\uptau} \end{array}\right] \\
&=& \left[\begin{array}{c|c} -A_J^{\uptau} & (t_2-1) B \\ \hline 0 & t_2 A_K-A_K^{\uptau}-(t_2-1)B^{\uptau}(A_J^{\uptau})^{-1} B \end{array}\right]
\end{eqnarray*}

Here $I$ denotes an identity matrix and $O$ denotes a matrix of all zeros. Since $J$ is fibered, $\det(-A_J^{\uptau})=(-1)^{2g}\det(A_J)=\pm 1$ (see \cite{bz}, Proposition 8.16). Consider now the $(2,2)$ entry on the right above. Taking the transpose and noting that $(A_J^{\uptau})^{-1}=(A_J^{-1})^{\uptau}$, we have that:
\[
((t_2 A_K-A_K^{\uptau})-(t_2-1)B^{\uptau}(A_J^{\uptau})^{-1} B)^{\uptau}=t_2 A_K^{\uptau}-A_K-(t_2-1)B^{\uptau}A_J^{-1} B.
\]
Now substitute $t_2=s^{-1}$, take the determinant, and multiply by $s^{2 g_K}=(-s)^{2g_K}$:
\begin{eqnarray*}
& & s^{2g_K} \det\left(s^{-1} A_K^{\uptau}-A_K-(s^{-1}-1)B^{\uptau}A_J^{-1} B\right) \\
&=& \det\left(-s \cdot (s^{-1} A_K^{\uptau}-A_K-(s^{-1}-1)B^{\uptau}A_J^{-1} B) \right) \\
                & = & \det\left(s A_K-A_K^{\uptau}-(s-1)B^{\uptau}A_J^{-1} B \right)\\
                & = & \det\left(s V^--V^+\right).
\end{eqnarray*}
Combining all these facts, we obtain $\overline{\Delta}_{\Sigma_K}(s)=\det(A_J) s^{2g_K} \nabla_{\Sigma_J,\Sigma_K}(0,s^{-1})$. This completes the proof of Theorem \ref{thm_main_alex}. \hfill $\square$

\begin{remark} To be consistent with the calculations of $\overline{\Delta}_{\upsilon}(t)$ in \cite{boden2015virtual} (see Example 7.7, page 28), one must use their convention for determining the the $\pm$ push-offs. The convention of \cite{boden2015virtual} is to use the \emph{left-hand rule}: the positive push-off is found by grabbing the knot with the \emph{\underline{\textbf{left hand}}} so that the thumb points in the direction of the knot's orientation and the fingers push through the Seifert surface toward $\Sigma_K^+$. Here, this convention must be used when computing the $\pm$ push-offs for both $\Sigma_K$ and the fiber $\Sigma_J$. However, if any consistent convention is used for both $\overline{\Delta}_{\Sigma_K}(t)$ and $\nabla_{\Sigma_J,\Sigma_K}(t_1,t_2)$ then the theorem will hold.
\end{remark}

\begin{figure}
\fcolorbox{black}{white}{
$\begin{array}{cc}
\begin{array}{c}
\def\svgwidth{1.75in}
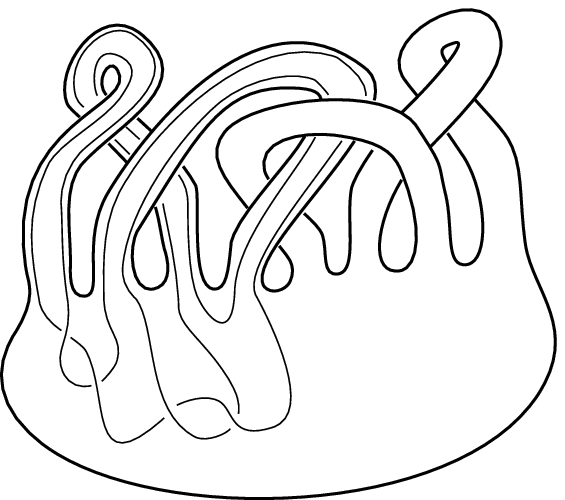 \end{array} &  \begin{array}{c}
\def\svgwidth{2.55in}
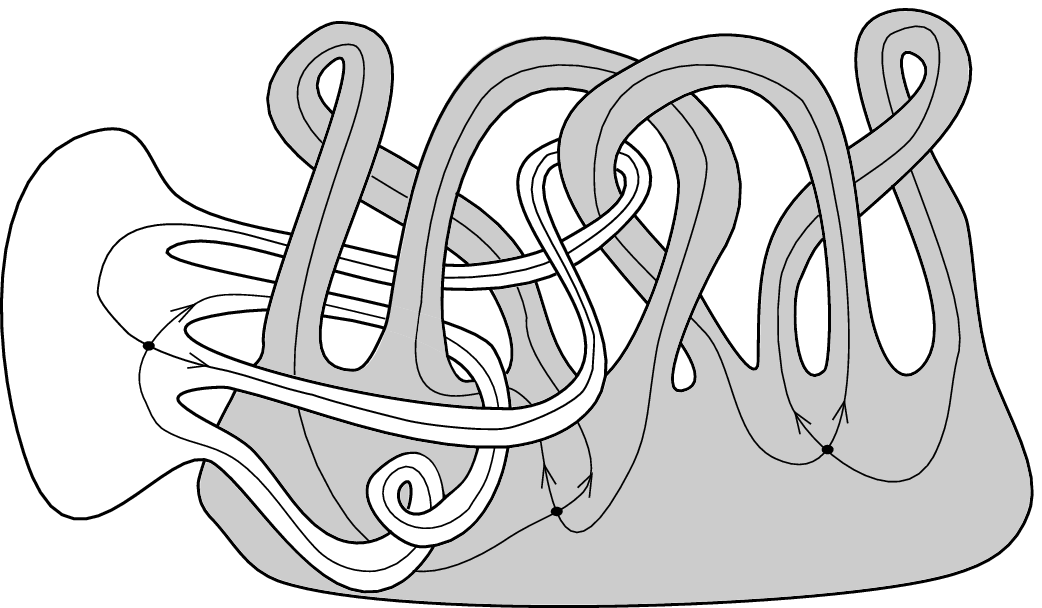\end{array}  \\ 
\end{array}
$}
\caption{(Left) A two component boundary link $L=J \sqcup K$ with $J \leftrightharpoons 8_{21}$. (Right) The basis used to compute the multi-variable Alexander polynomial.}\label{fig_alex_ex}
\end{figure}

\begin{example} Let $L=J\sqcup K$ be the two component link on the left in Figure \ref{fig_alex_ex}. Here, $J$ is the fibered knot $8_{21}$. A minimal genus Seifert surface $\Sigma_J$ of $J$ (i.e. a fiber) is depicted in disc-band form. On the right in Figure \ref{fig_alex_ex}, we have $\Sigma_J \sqcup \Sigma_K$, where $\Sigma_K$ is a Seifert surface for $K$. Curves representing symplectic bases for $H_1(\Sigma_J;\mathbb{Z})$ and $H_1(\Sigma_K;\mathbb{Z})$ are given. From this, one can compute the MVAP as previously defined. The associated virtual knot $\upsilon \leftrightharpoons 4.105$ \cite{vtable} of $L$ can be seen in Figure \ref{fig_alex_gauss}. This agrees with the computation of $\overline{\Delta}_{\upsilon}(t)$ in \cite{boden2015virtual}, Table 2.
\end{example}

\begin{eqnarray*}
\nabla_{\Sigma_J,\Sigma_K}(t_1,t_2)&=& -2 + 8 t_1 - 10 t_1^2 + 6 t_1^3 - t_1^4 + 2 t_2 - 10 t_1 t_2 + 15 t_1^2 t_2\\ 
             &-& 10 t_1^3 t_2 + 2 t_1^4 t_2 - t_2^2 + 6 t_1 t_2^2 - 10 t_1^2 t_2^2 +8 t_1^3 t_2^2 - 2 t_1^4 t_2^2.
\end{eqnarray*}
\begin{eqnarray*}
\det(A_J) \cdot t^2 \cdot \nabla_{J,K}(0,t^{-1})&=& (-1) t^2 \left(-2-\frac{1}{t^2}+\frac{2}{t} \right)\\
                  &=& 1-2t+2t^2 \\
                  &\doteq& \overline{\Delta}_{\upsilon}(t)
\end{eqnarray*}

\begin{figure}
\fcolorbox{black}{white}{
$\begin{array}{ccc} \\
\begin{array}{c}
\def\svgwidth{2in}
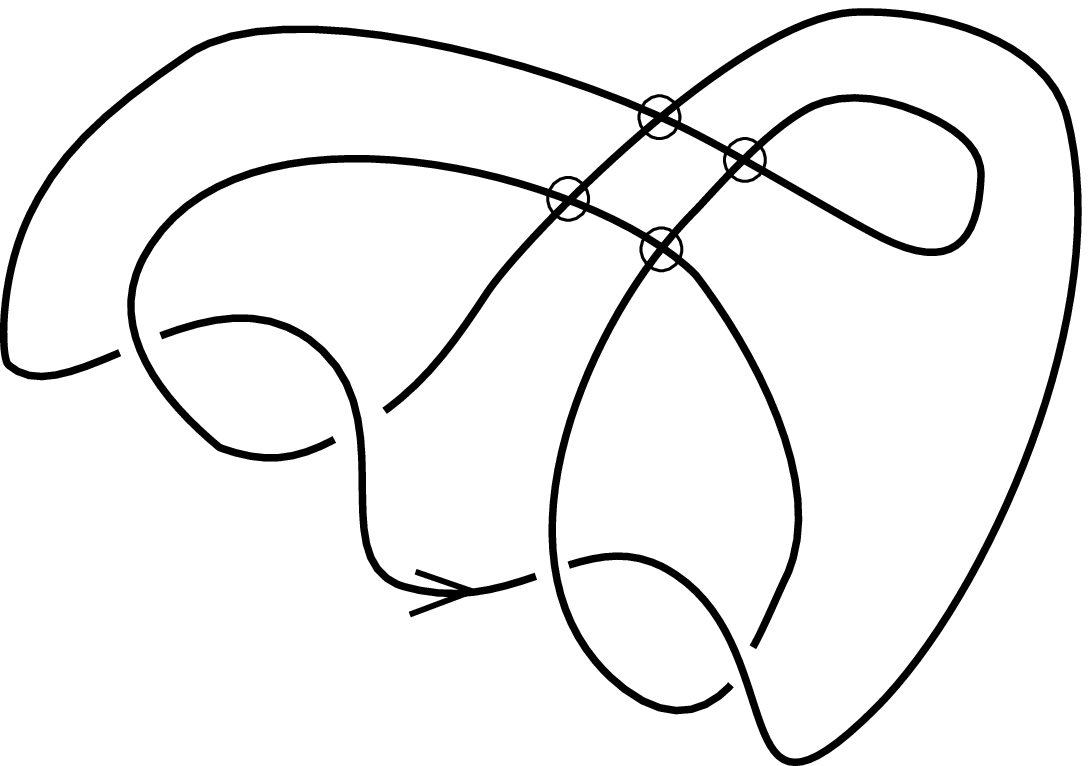 \end{array} &  & \begin{array}{c}
\def\svgwidth{1.8in}
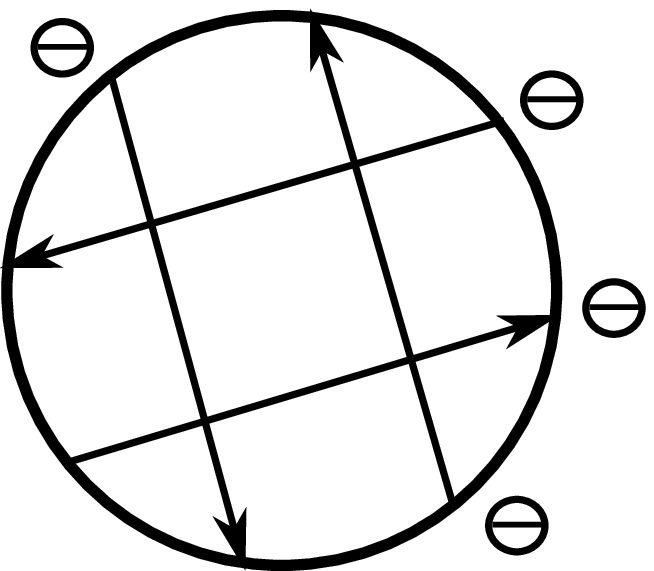\end{array}  \\ \\
\end{array}
$}
\caption{(Left) A virtual knot diagram of 4.105. (Right) A Gauss diagram of 4.105.}\label{fig_alex_gauss}
\end{figure}

\section{Relating the index and the triple linking number} \label{sec_index}

\subsection{Theorem statement} As mentioned in Section 1.3, almost classical knots are homologically trivial in some thickened surface. The two equivalent combinatorial definitions of AC are: (1) it has a diagram which is Alexander numerable, or (2) it has a diagram in which every classical crossing has index zero. The index zero definition of AC stands out as it is a collection of local conditions as opposed to a single global one. Here we give an interpretation of the index of a crossing in terms of the Milnor triple linking number. Hence we obtain relations with classical link invariants in the case that the associated virtual knot is not AC. In this section, we will prove the following theorem.

  \begin{theorem}
  \label{CochranTripleLinking}
  Let $L=J \sqcup K$ be an SSF link, where $J$ is a connected sum of trefoils, $\Sigma_J$ the fiber, and $\upsilon$ the associated virtual knot. Let $x$ denote a classical crossing of $\upsilon$ and also the corresponding crossing on $\Sigma_J$ (see Section \ref{sec_vc}). Let $L_{x}=J \sqcup K_{2} \sqcup K_{3}$ be the link in $\mathbb{S}^3$ obtained by performing the oriented smoothing at $x$. Then:
  $$ \bar{\mu}_{123}(L_x)=-\text{Index}(x)\mod{(lk(K_{2},K_{3}))}.$$
  \end{theorem} 
    
 We will break the proof down into several parts. First we will show that any virtual knot  can be represented as a diagram on a surface that is connected sum of trefoil fibers. Thus, classical crossing of the associated virtual knot correspond to crossings on the surface (see Section \ref{sec_vc}). Next we will review the index of a crossing (see Section \ref{sec_index_cross}) and two methods for computing $\bar{\mu}_{123}$: Cochran's derivative of links (see Section \ref{sec_deriv}) and the method of Mellor-Melvin (see Section \ref{sec_melmel}). The proof of Theorem \ref{CochranTripleLinking} is in Section \ref{sec_trip_proof}, along with an example. 

\subsection{Virtual knots on sums of trefoils} \label{sec_on_tref} Every virtual knot diagram can be represented on a connected sum of trefoil or figure-eight fibers. Here we prove this fact and discuss how it is related to the proof of Theorem \ref{CochranTripleLinking}.  

\begin{lemma} \label{lemma_trefoil}
   Every virtual knot diagram $\upsilon$ can be associated to some two component SSF link $L=J\sqcup K$, where $J$ is a connected sum of trefoils or a connected sum of figure-eights, and $\Sigma_J$ is a boundary connect sum of fibers of such knots.
   \end{lemma}
   \begin{proof} Every virtual knot diagram can be represented as a diagram $D$ on its Carter surface $\Sigma$. Let $\Sigma'=\Sigma \smallsetminus \nu(z_0)$ where $z_0 \in \Sigma$ is disjoint from $D$. Let $T$ be a fiber of a trefoil or figure-eight knot. Then $T$ has genus $1$. If the genus $g$ of $\Sigma$ is $0$, then we may embed $D$ as a small disc on $T$. If $g\ge 1$, then $\Sigma'$ is diffeomorphic to a boundary connected sum of copies of $T$ (see Figure \ref{fig_trefoils}). Taking the image of $D$ under this diffeomorphism implies the result.
   \end{proof}

\begin{figure}[htb]
\fcolorbox{black}{white}{
$\begin{array}{c} \\
\begin{array}{c}
\def\svgwidth{4.5in}
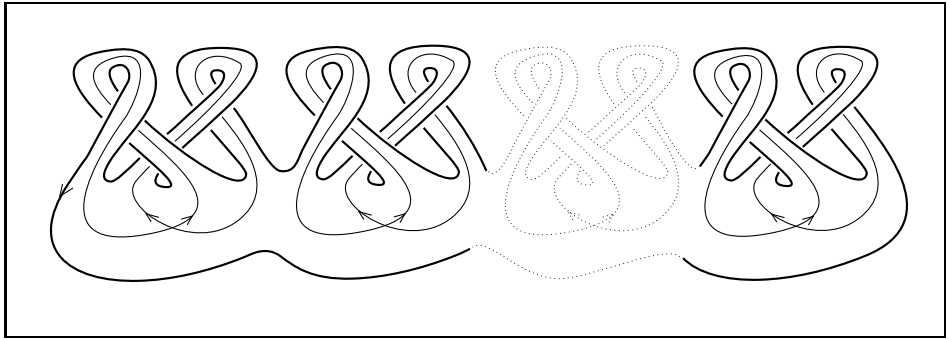 \end{array}  \\ \\
\end{array}
$}
\caption{A boundary connected sum of right handed trefoil fibers.}\label{fig_trefoils}
\end{figure}

The reason to represent a virtual knot $\upsilon$ as in Lemma \ref{lemma_trefoil} is that we need to choose the Seifert matrix of $J$ so that it behaves nicely with respect to the matrix of the intersection form of $\Sigma_J$. Indeed, suppose that $\Sigma_{J}$ is a boundary connect sum of trefoil fibers as in Figure \ref{fig_trefoils}. Then with respect to a suitable basis, the Seifert matrix $A_J$ for $\Sigma_J$ is given by: 
\[
A_J=\text{diag}(H,\ldots,H), \,\, \text{where } H=\left[\begin{array}{cc} -1 & 1 \\ 0 & -1 \end{array}\right].
\]

Now recall that $A_J-A_J^{\uptau}=F$, where $F$ is a matrix for the intersection form on $\Sigma_J$. Then $F$ is also a block diagonal matrix such that all blocks are given by $\left[\begin{array}{cc} 0 & 1 \\ -1 & 0 \end{array}\right]$.  On the other hand, we may consider $F$ as the matrix for  the linear transformation $\mathcal{F}:H_{1}(\Sigma_{J})\rightarrow H_{1}(\mathbb{S}^{3}\smallsetminus\nu\Sigma_{J})$ with respect to this same basis and its dual. Since the Alexander polynomial evaluated at 1 gives $\pm1$, $F$ is an invertible matrix over $\mathbb{Z}$ and thus $\mathcal{F}$ is an isomorphism. Furthermore, it is straightforward to verify a special relationship between the Seifert matrix and the matrix of the intersection form $F$ for a connected sum of trefoils: $A_J^{\uptau}FA_J=F$. Also note that $F^{-1}=-F=F^{\uptau}$.  The same relation holds for a connected sum of figure-eight knots. Henceforth, we will only consider the case of connected sums of trefoil fibers in order to avoid needless notations and special case arguments.

\subsection{The index of a crossing} \label{sec_index_cross}
The index has been used by many authors as an ingredient in virtual knot invariants (see e.g. \cite{henrich2010sequence}).
Let $x$ be a classical crossing of an oriented virtual knot diagram $\upsilon$. Smooth this crossing in the orientation preserving fashion to yield a two component virtual link. Then flatten all classical crossings to obtain two flat virtual knot diagrams $\upsilon_{1}$ and $\upsilon_{2}$. Let $\upsilon_{1}$ be the component to the left of the smoothed crossing. Each double point (flattened classical crossing) between the two components can be given a sign $\varepsilon_i$, as in Figure \ref{IndexRule}. The index of the crossing $x$ is defined to be:
\[
\emph{Index}(x)=\sum_i \varepsilon_i.
\]
\begin{figure}
\fcolorbox{black}
{white}{
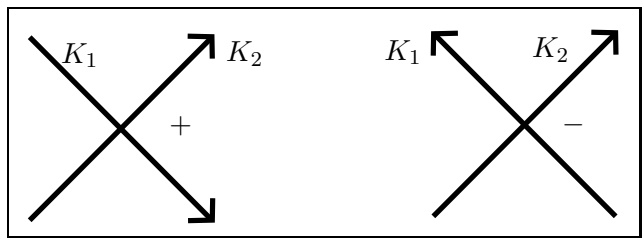
}
\caption{A right hand rule for computing the sign of flattened crossing.}
\label{IndexRule}
\end{figure}

In the context of virtual covers, the index can be computed via the intersection form as follows. Let $L=J \sqcup K$ be an SSF link with associated virtual knot $\upsilon$. Let $x$ be one of the classical crossings of $\upsilon$, that is, $x$ is a crossing in a ball (see Figure \ref{fig_cross}) of $K$ in SSF on $\Sigma_{J}$. Construct the three component link $J \sqcup K_{1} \sqcup K_{2}$ by smoothing at $x$ in the orientation preserving fashion and ordering the components so that $K_{1}$ is the left half. Notice that each of $K_{1}$ and $K_{2}$ are in SSF with respect to $J$. Let $k_{1}$ and $k_{2}$ be the corresponding diagrams of $K_{1}$ and $K_{2}$ on $\Sigma_{J}$. As noted in \cite{chrisman2014three}, $\emph{Index}(x)$ can be computed via the intersection form on $\Sigma_{J}$ as $\emph{Index}(x)=[k_{1}]\cdot [k_{2}]$ where $[k_{i}] \in H_{1}(\Sigma_{J};\mathbb{Z})$ (e.g see \cite{rolfsen1976knots}, page 202). Note that changing the order changes the sign of $\emph{Index}(x)$ and hence our definition here sometimes differs from other authors by a sign.

\subsection{Milnor's Triple linking numbers} \label{sec_milnor}
Milnor in \cite{milnor1957isotopy} extended the concept of linking numbers to links of more than two components. In general, these higher order linking numbers are defined when all lower order linking numbers are null. For example, the triple linking number of the Borromean rings is $\pm1$. In the case that the lower order linking numbers are non-zero then those of higher order are defined up to an indeterminacy.
    
A casual description of Milnor's higher linking numbers (as described in \cite{mellor2003geometric}) goes like this: let $L$ be an $k$-component link with components $L_{i}$, $\pi$ its fundamental group, and $\pi^{n}$ the $n^{th}$ term in its lower central series. Let $l_{i}$ and $m_{i}$ be the longitude and meridian respectively of component $L_{i}$. In this case, $l_{i}$ can be written as a word $l_{i}^{n}$ in the preferred meridians $m_{1}, \ldots, m_{k}$, taken modulo $\pi^{n}$. Take the Magnus expansion of this word by substituting $(1+h_{i})$ for $m_{i}$ and $(1-h_{i}+h_{i}^{2}\ldots)$ for $m_{i}^{-1}$. For any sequence $i_{1}i_{2}\ldots i_{r}$ of integers, with each term between 1 and $k$, let $\bar{\mu}_{i_{1}i_{2}\ldots i_{r}}$ be the coefficient of $h_{i_{1}}\ldots h_{i_{r-1}}$ in the Magnus expansion of the word $l_{i_{r}}^{n}$. This is defined only up to the greatest common divisor of all the lower order linking numbers. 

Here we will only consider Milnor's triple linking number $\bar{\mu}_{123}$. Fortunately, there are several results that allow one to compute the triple linking number for a three component link using geometric constructions. In particular, Cochran \cite{cochran1985geometric} used his notion of link derivatives to compute Milnor's triple linking number when the pairwise linking numbers are zero. Mellor and Melvin \cite{mellor2003geometric} give a geometric method of computing $\bar{\mu}_{123}$ that can be used in general. Both methods will be used in the proof of Theorem \ref{CochranTripleLinking}.  

\subsection{Method 1: Derivatives of Links} \label{sec_deriv} Let $L=J \sqcup K$ be any two component link such that $\text{lk}(J,K)=0$. Suppose that $\Sigma_{J}$ and $\Sigma_{K}$ are Seifert surfaces for $J,K$, respectively. Since $\text{lk}(J,K)=0$, we may assume that $\Sigma_J\cap K=\emptyset$ and $\Sigma_K \cap J=\emptyset$. If required, this can be done by starting with any pair of Seifert surfaces and ``tubing out'' intersections of $J$ with $\Sigma_K$ and $K$ with $\Sigma_J$ (see \cite{cochran1990derivatives}, Appendix A. I). A \emph{derivative} of $L$ is the curve $c(J,K)=\Sigma_{J} \cap \Sigma_{K}$. This curve may be chosen so that it is a single connected component \cite{cochran1985geometric}. It is oriented so that the triple ($c(J,K)$, positive normal to $\Sigma_{J}$, positive normal to $\Sigma_{K}$) agrees with the standard orientation of $\mathbb{S}^{3}$. The following theorem, due to Cochran, allows us to compute $\bar{\mu}_{123}$ from the linking number of two curves.

\begin{theorem}(see \cite{cochran1990derivatives}, Appendix B.II) \label{thm_tim}
Suppose that $L=K_{1} \sqcup K_{2} \sqcup K_{3}$ is a three component link such that each pairwise linking number is zero. Furthermore, let $\Sigma_{1}$ and $\Sigma_{2}$ be Seifert surfaces for $K_{1}$ and $K_{2}$ such that $\Sigma_i \cap K_{3-i}=\emptyset$ for $i=1,2$. Let $c(K_{1},K_{2})$ be the derivative of $K_{1} \sqcup K_2$. Then: 
\[
\bar{\mu}_{123}(L)=-lk(c(K_{1},K_{2}),K_{3}).
\]
\end{theorem}

Link derivatives will appear in our homological calculations ahead. Note that we may think of $c(J,K)$ as an element of $H_{1}(\Sigma_{J};\mathbb{Z})$. Let $A_J$ be a Seifert matrix for $\Sigma_{J}$ and let $B=(A_J-A_J^{\uptau})^{-1}$ as matrices. We can think of $B$ as the matrix that represents $\mathcal{F}^{-1} :H_{1}(\mathbb{S}^{3}\smallsetminus\nu\Sigma_{J};\mathbb{Z}) \rightarrow H_{1}(\Sigma_{J};\mathbb{Z})$ with respect to the appropriate basis (see Section \ref{sec_on_tref}). Recall that $A_J$ and $A_J^{\uptau}$, considered as the matrices for linear transformations of homology groups, are the negative and positive push-offs of $\Sigma_{J}$, respectively. Since we may assume that $\Sigma_{J}$ and $\Sigma_{K}$ are in general position, the curve $c(J,K)$ on $\Sigma_{J}$ can be pushed off so that its positive and negative push-offs lie on $\Sigma_{K}$. Now observe that cutting $\mathbb{S}^{3}$ along $\Sigma_{J}$ cuts $\Sigma_{K}$ so that its boundary consists of three components: $K$ and the positive and negative push-offs of $c(J,K)$. Then as elements of $H_{1}(\mathbb{S}^{3}\smallsetminus\nu\Sigma_{J})$, we have that: 
\[
[K]=(A_J-A_J^{\uptau})[c(J,K)].
\]
We record this fact as a lemma for future reference (see also \cite{cochran1985geometric}, Section 8).
    
\begin{lemma} \label{lemma_B_deriv}
Suppose that $L=J \sqcup K$ and $lk(J,K)=0$. Let $\Sigma_{J}$, $\Sigma_{K}$ be Seifert surfaces such that $J=\partial \Sigma_J$, $K=\partial \Sigma_K$, $\Sigma_J \cap K=\emptyset$, and $\Sigma_K \cap J=\emptyset$. Let $[c(J,K)] \in H_{1}(\Sigma_{J};\mathbb{Z})$ be the homology class of a derivative of $L$. Furthermore, let $B=(A_J-A_J^{\uptau})^{-1}$ be the matrix that represents $\mathcal{F}^{-1}:(\mathbb{S}^{3}\smallsetminus \nu\Sigma_{J};\mathbb{Z}) \rightarrow H_{1}(\Sigma_{J};\mathbb{Z})$, where $A_J$ is a Seifert matrix of $J$. Then $B[K]=[c(J,K)]$.
\end{lemma}
      
\subsection{Method 2: Counting local intersections} \label{sec_melmel}  The second method, due to Mellor-Melvin \cite{mellor2003geometric}, involves counting two combinatorial objects: triple points and intersections between a component of the link and the other two Seifert surfaces. Consider the link $L=K_{1}\sqcup K_{2}\sqcup K_{3}$ such that each component has a Seifert surface $\Sigma_{i}$ in general position with respect to the others. Let $\Sigma=\Sigma_1 \cup \Sigma_2 \cup \Sigma_3$. Define $t_{123}(\Sigma)=\#(\Sigma_{K_{1}}\cap \Sigma_{K_{2}} \cap \Sigma_{K_{3}})$ counted with sign so that a triple point is positive if and only if the ordered basis of normal vectors to $\Sigma_{K_{1}},\Sigma_{K_{2}},\Sigma_{K_{3}}$ agree with the standard orientation of $\mathbb{S}^{3}$.

Next, we count component-surface intersections. Chose a base point on each component. Then for each  $K_{i}$ we build a word $w_{i}$ in the letters $\{1^{\pm},2^{\pm},3^{\pm}\}$ as follows: from the base point walk in the direction of the orientation of the component and record with sign the component whose Seifert surface you intersect. Notice that $w_{i}$ will be a word in $\{1^{\pm},2^{\pm},3^{\pm}\}\smallsetminus\{i^{\pm}\}$.  Once we have a particular word $w_{i}$, its Magnus expansion is found by substituting $j^{+} \to (1+h_{j})$ and 
  $j^{-} \to (1-h_{j}+h_{j}^{2}- \ldots)$. Let $e_{ijk}$ be the coefficient of the word $h_{i}h_{j}$ in $w_{k}$, where 
  $i,j,k \in \{1,2,3\}$ are distinct. Define $m_{123}(\Sigma)=e_{123}+e_{231}+e_{312}$. 
   
\begin{theorem}[Mellor-Melvin \cite{mellor2003geometric}] \label{thm_melmel} Consider the link $L=K_{1} \sqcup K_{2} \sqcup K_{3}$ with Seifert surfaces as described above. Then:
\[
\bar{\mu}_{123}(L) \equiv m_{123}(\Sigma)-t_{123}(\Sigma)\ mod(\delta),
\]
where $\delta$ is the greatest common divisor of the pairwise linking numbers of $K_{1},K_{2}$, and $K_{3}$.
\end{theorem}
      
\subsection{Proof of Theorem \ref{CochranTripleLinking}} \label{sec_trip_proof} Now let us return to the setting of virtual covers and the proof of Theorem \ref{CochranTripleLinking}. Let $L=J\sqcup K$ be a link where $J$ is fibered and $\text{lk}(J,K)=0$. Let $A_J$ be a Seifert matrix for a fiber $\Sigma_{J}$ of $J$. Furthermore, suppose that $L$ is an SSF link and let $\upsilon$ be the associated virtual knot. Let $x$ be crossing in a ball of $K$, as in Figure \ref{fig_cross}. Let $L_{x}=J \sqcup K_{2} \sqcup K_{3}$ be the link formed by smoothing $K$ at $x$ in the orientation preserving fashion, again ordering the components so that $K_{2}$ is on the left and $K_{3}$ is on the right at $x$. For notational purposes write $J=K_{1}$. Let $k_{2},k_{3}$ be the diagrams of $K_2,K_3$ on $\Sigma_{J}$, respectively. In terms of homology, $A_{J}[k_{2}]=[K_{2}]$ and  $A_{J}[k_{3}]=[K_{3}]$ in $H_{1}(\mathbb{S}^{3}\smallsetminus \nu\Sigma_{J};\mathbb{Z})$. Recall from Section \ref{sec_index} that $\emph{Index}(x)=[k_{2}]\cdot [k_{3}]$.
  
  \begin{lemma}\label{lemma_milnor_next}
  In the situation described above, $m_{123}(\Sigma)=0$.
  \end{lemma}
  \begin{proof}
  Since $L$ is in SSF, it follows that both $J \sqcup K_2$ and $J \sqcup K_3$ are in SSF and $\text{lk}(J,K_{2})=\text{lk}(J,K_{3})=0$. Thus, we may assume that $J$ does not intersect the Seifert surface for $K_{2}$ nor the Seifert surface for $K_{3}$. It follows that $w_{1}$ is the empty word and so $e_{231}=0$. With regards to $K_{2}$ and $K_{3}$ we must consider two cases: $\text{lk}(K_{2},K_{3})=0$ or $\text{lk}(K_{2},K_{3})\neq 0$. In the case that $\text{lk}(K_{2},K_{3})=0$, we may assume again that the Seifert surface for $K_{2}$ does not intersect $K_{3}$ and vice versa. Again we see that $w_{2}$ and $w_{3}$ are the empty words. Thus $e_{123}=e_{312}=0$. In the case that $\text{lk}(K_{2},K_{3})\neq 0$, notice that $w_{2}$ is a word in only the letters $3^{\pm}$ and $w_{3}$ is a word only in the letters $2^{\pm}$. In both of these cases the Magnus expansions of these words contain only powers of a single variable and we may again conclude that $e_{123}=e_{312}=0$. Thus the claim follows. 
  \end{proof}
  
The previous lemma tells us that the triple linking number can be computed as a signed sum of triple points and in fact that $\bar{\mu}_{123}(L_x)=-t_{123}\pmod \delta$. Moreover, as $\text{lk}(J,K_{2})=\text{lk}(J,K_{3})=0$, it follows that $\delta=\text{lk}(K_{2},K_{3})$.
  
\begin{lemma} \label{lemma_milnor_last} For $L_x=J \sqcup K_{2} \sqcup K_{3}$ as above, coming from smoothing at a crossing $x$,
\[
  \emph{Index}(x)=t_{123}(L).
\]
\end{lemma}
  
  \begin{proof}
  Think of the derivatives $c(J,K_{2})$ and $c(J,K_{3})$ as curves that lie on the fiber $\Sigma_{J}$. Observe that the double points of these two derivatives are the triple points of the intersection of the three Seifert surfaces. Recall that a triple point of the intersection is positive when in its neighborhood $(\Sigma_{J},\Sigma_{K_{2}},\Sigma_{K_{3}})$ has the standard right handed orientation. Let $n_{i}$ be the positive normal to $\Sigma_{K_{i}}$. Then the derivative $c(J,K_{i})$ is oriented so that $(c(J,K_{i}),n_{1}, n_{i})$ has the standard right handed orientation \cite{cochran1990derivatives}. It is not hard to determine that $t_{123}(L)=[c(J,K_{2})]\cdot [c(J,K_{3})]$ (according to the standard right hand rule for the intersection form on $\Sigma_{J}$). Using Lemma \ref{lemma_B_deriv} and the discussion of Section \ref{sec_on_tref}, we have:
   \begin{eqnarray*}
  t_{123}&=&[c(J,K_{2})] \cdot [c(J,K_{3})]\\
  &=&(B[K_{2}])^{\uptau}F(B[K_{3}])\\
  &=&(BA_{J}[k_{2}])^{\uptau}F(BA_{J}[k_{3}])\\
  &=&[k_{2}]^{\uptau}A_{J}^{\uptau}B^{\uptau}FBA_{J}[k_{3}]\\
  &=&[k_{2}]^{\uptau}A_{J}^{\uptau}(F^{\uptau})^{\uptau}A_{J}[k_{3}]\\
  &=&[k_{2}]^{\uptau}A_{J}^{\uptau}FA_{J}[k_{3}]\\
  &=&[k_{2}]^{\uptau}F[k_{3}]\\
  &=&[k_{2}]\cdot [k_{3}]\\
  &=&\emph{Index}(x). 
  \end{eqnarray*}
  \end{proof}

The proof of Theorem \ref{CochranTripleLinking} now follows by combining Lemmas \ref{thm_melmel}, \ref{lemma_milnor_next}, and \ref{lemma_milnor_last}. We conclude this section with an example computation.

\begin{example} Consider the SSF link $L=J\sqcup K$ on the left in figure \ref{TonT1}. The virtual trefoil is the associated virtual knot of this link (exercise). The crossing labeled $x$ has index $-1$ as can be computed using the intersection form $[k_{2}]\cdot [k_{3}]$. Notice also that $\text{lk}(K_{2},K_{3})=0$. By Cochran's result (Theorem \ref{thm_tim}), $\bar{\mu}_{123}(L_x)=1$. Figure \ref{TonT2} shows Seifert surfaces for each component of the link, along with a derivative $c(J,K_{2})$. It can readily be seen that $\text{lk}(c(J,K_{2}),K_{3})=-1$. Thus applying theorem \ref{CochranTripleLinking} we can observe that $-\emph{Index}(x)=\bar{\mu}_{123}(L_x)=-\text{lk}(c(J,K_{2}),K_{3})=1$. 
\end{example}

  \begin{figure}
  \fcolorbox{black}{white}{
  \def\svgwidth{4.5in}
  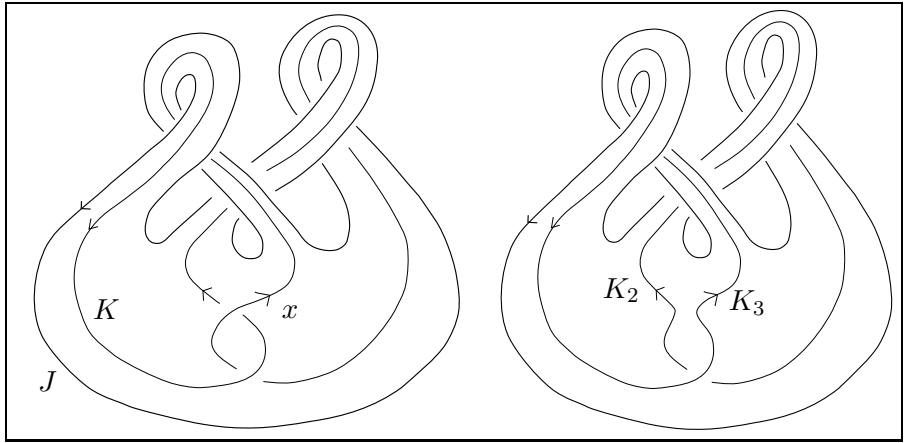}
  \caption{(Left) The link $J \sqcup K$ in SSF. (Right) The three component link found by smoothing at $x$}
  \label{TonT1}
 \end{figure}
  
  \begin{figure}
  \fcolorbox{black}{white}{
  \def\svgwidth{3in}
  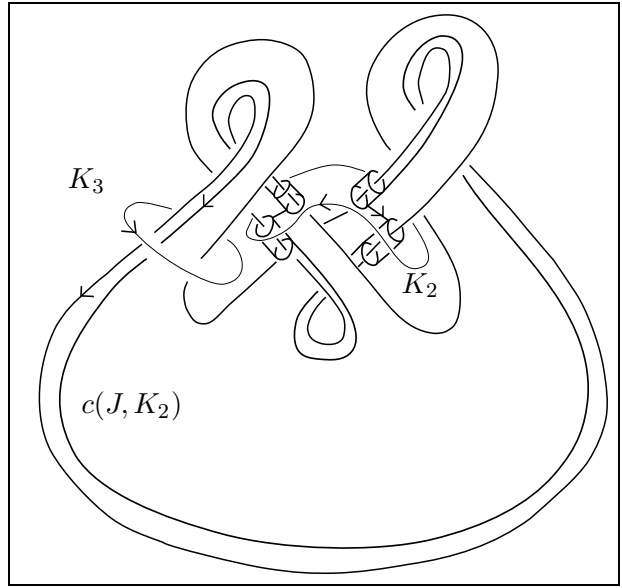}
  \caption{The Seifert surfaces for each component. Two tubes are added to $\Sigma_{K_{2}}$, only portions of which are show. The derivative appears as the intersection of this tubed surface as shown.}
  \label{TonT2}
  \end{figure}
  
\section{Fiber Stabilization of Links} \label{sec_fiber} 

\subsection{Theorem statement and its motivation} The previous sections have shown how virtual covers can be used to relate virtual knot and classical link invariants for certain families of links. In this section we introduce \emph{fiber stabilization of links}. This extends virtual covers to all multi-component links. The extension improves upon \cite{chrisman2014virtual}, where virtual covers were extended to links $J \sqcup K$, with $K$ a knot and $J$ a virtually fibered link (in the sense of Thurston). 

\begin{definition} Let $L=J \sqcup K$ be an $(n+1)$-component link with $K$ a knot. A \emph{fiber stabilization} of $L$ is the addition of an unkotted component $J_0$ to $L$ so that the sublink $J_0 \sqcup J$ is a fibered link and $K$ has algebraic intersection number $0$ with some fiber $\Sigma_0$ of $J_0 \sqcup J$. The fiber stabilized link is denoted $L_0=J_0 \sqcup J \sqcup K$.
\end{definition}

\begin{figure}[htb]
\fcolorbox{black}{white}{
\begin{tabular}{ccc}
\begin{tabular}{c} \def\svgwidth{1.3in} 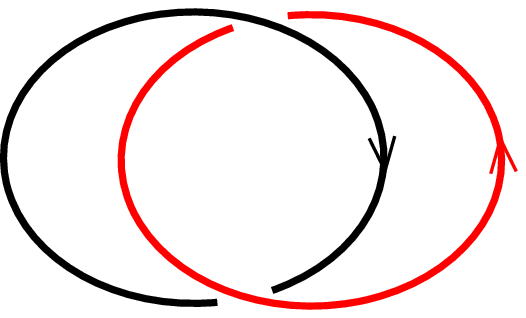 \end{tabular} & \begin{tabular}{c}  \def\svgwidth{1.3in} 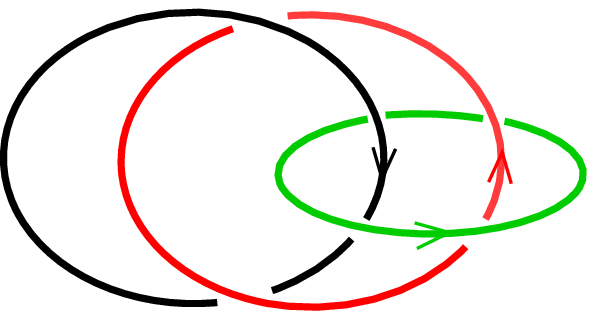 \end{tabular} & \begin{tabular}{c} \def\svgwidth{1.3in} 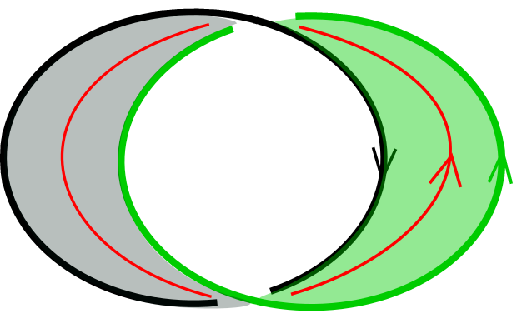 \end{tabular}
\end{tabular}}
\caption{A fiber stabilization of $L=J \sqcup K$ by adding an unknotted component $J_0$. A fiber $\Sigma_0$ of $J_0 \sqcup J$ is given on the right.} \label{fig_hopf}
\end{figure}

A simple example of a fiber stabilization is given in Figure \ref{fig_hopf}. Fiber stabilization can be viewed as a generalization of the technique used in \cite{chrisman2014virtual} to give a virtual knot theory proof of the classical result of Whitten that for every $m$, there are infinitely many $m$ component non-invertible non-split links whose components are all invertible (see \cite{whitten_sublinks}, where a stronger result is proved). The idea is to begin with a two component SSF link $L=J \sqcup K$ with $J$ and $K$ each invertible in $\mathbb{S}^3$. The knot $K$ is chosen so that the associated virtual knot $\upsilon$ is non-classical and fails to satisfy a certain symmetry condition that would necessarily hold if $L$ was invertible (see \cite{chrisman2014virtual}, Theorem 5). The fact that $\upsilon$ is non-classical guarantees that $L$ is non-split (see \cite{chrisman2014virtual}, Corollary 4). A three component link $L_0=J_0 \sqcup J \sqcup K$ is obtained from $L$ by adding an unknotted component so that the link $J_0\sqcup J$ has a connected Seifert surface consisting of a boundary connect sum of a fiber $\Sigma_J$ of $J$ and a fiber of the Hopf link. Then $J _0\sqcup J$ is a fibered link, $L_0$ is in SSF, and $\upsilon$ is the associated virtual knot to $L_0$. It follows that $L_0$ is a three component non-invertible non-split link whose components are all invertible. The argument can be iterated to obtain a non-invertible non-split link of any number of components, all of which are invertible. This example shows that fiber stabilization adds new tools to the study of classical links.

A second motivation to study fiber stabilization of links is that families of link invariants often determine the invariants of their sublinks. The Torres conditions, for example, relate the MVAP of an arbitrary link to that of its sublinks (see \cite{kawauchi}, Theorem 7.4.1). As another example, if all the Milnor invariants of a link are known, then so are all the Milnor invariants of its multi-component sublinks. Since a link $L$ is a sublink of its fiber stabilization $L_0$, it is thus natural to ask if invariants of an associated virtual knot $\upsilon$ of $L_0$ can be related to invariants of $L$. For the moment we will leave these aspirations aside and content ourselves with proving the following theorem.

\begin{theorem} \label{thm_fsl} Every multi-component link has a fiber stabilization and every fiber stabilized link has a virtual cover.
\end{theorem}

Applications of this idea will be explored elsewhere. To prove that every multi-component link has a fiber stabilization, we will use Stallings observation that every link is a sublink of a fibered link \cite{stallings}.  Section \ref{sec_braid_theory} will review this result in detail, along with some other necessary constructions from braid theory. In Section \ref{sec_stable_proofs}, we prove Theorem \ref{thm_fsl}.

\subsection{Some braid theory} \label{sec_braid_theory} Two constructions from braid theory are needed to prove that every link has a fiber stabilization: mixed braids and homogeneous braids. Let $B$ be a braid and let $\widehat{B}$ denote the closure of $B$. Let $L$ be any link in $\mathbb{S}^3\smallsetminus\nu \widehat{B}$. Then $\widehat{B} \sqcup L$ is called a \emph{mixed link}. Here $\widehat{B}$ is called the \emph{fixed part} and $L$ is called the \emph{moving part}. The mixed link may be represented by a \emph{mixed braid} $B \sqcup \beta$ whose closure $\widehat{B} \sqcup \widehat{\beta}$ is equivalent to $\widehat{B} \sqcup L$ \cite{lamb_rourke}. 

A mixed braid on $m+n$ strands is said to be \emph{parted} if the first $m$ strands are the strands in the fixed part and the last $n$ strands are in the moving part. In \cite{lamb_rourke} (Section 2, Lemma 1), it was shown that every mixed link may be represented by a parted mixed braid. Moreover, every parted mixed braid may be combed so that the fixed part lies below the moving part. This means that the mixed braid may written as $\beta'\cdot B'$, where the first $m$ strands of $\beta'$ are trivial and $B'$ is $B$ with the trivial $n$-strand braid attached on the right (see Figure \ref{fig_comb}).

\begin{figure}[htb]
\fcolorbox{black}{white}{
\begin{tabular}{c}
\def\svgwidth{2in}
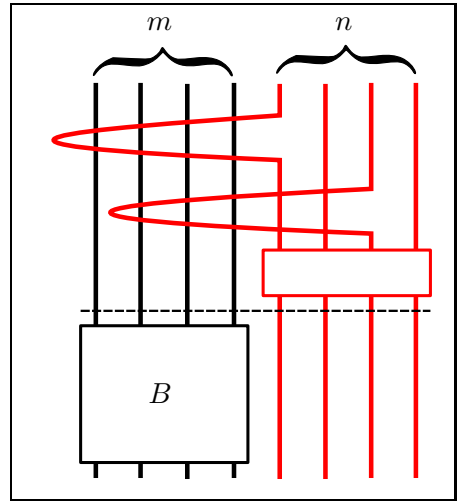 
\end{tabular}}
\caption{A schematic diagram of a parted combed braid on $m+n$ strands.} \label{fig_comb}
\end{figure}

Homogeneous braids, our second ingredient, are useful for constructing fibered links. Let $n$ be a fixed number of strands and let $\sigma_i$ denote the $n$-strand braid where the $i$-th strand over-crosses the $(i+1)$-st stand. Suppose furthermore that $B$ is written in the generators $\sigma_1,\ldots,\sigma_n$ as $B=\sigma_{i_1}^{\varepsilon_{i_1}}\sigma_{i_2}^{\varepsilon_{i_2}}\cdots\sigma_{i_N}^{\varepsilon_{i_N}}$, where $\varepsilon_{i_j}=\pm 1$. If every $\sigma_i$ appears in the word for $B$ and $\varepsilon_{i_j}=\varepsilon_{i_k}$ whenever $\sigma_{i_j}=\sigma_{i_k}$, then $B$ is said to be \emph{homogeneous}. To see that one may construct a fibered link from a homogeneous braid, proceed as follows (Stallings \cite{stallings}). First note that the closure of the trivial $n$-strand braid can be viewed as a union of parallel squares bounding $n$ discs. Attach half-twisted bands to these discs according to the word $B$. The resulting surface is a successive Murasugi sum (or \emph{generalized plumbing}) of fibers of the fibered links $\widehat{\sigma_1^{a_1}},\ldots,\widehat{\sigma_{n-1}^{a_{n-1}}}$, for some non-zero integers $a_1,\ldots,a_{n-1}$. As Murasugi sums of fibered links are fibered, $\widehat{B}$ is fibered (see e.g. \cite{kawauchi}, Section 4.2). Figure \ref{fig_homogenous} shows the construction for the homogeneous $3$-strand braid $\sigma_2^{-1}\sigma_1\sigma_2^{-1}\sigma_1$.

\begin{figure}[htb]
\fcolorbox{black}{white}{
\begin{tabular}{ccccc} \\
\begin{tabular}{c} \def\svgwidth{1in} 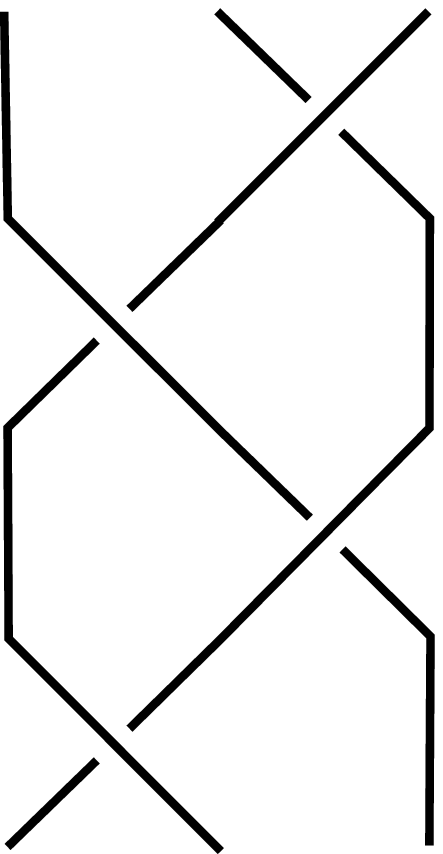 \end{tabular} & & $\longrightarrow$ & & \begin{tabular}{c}  \def\svgwidth{1.6in} 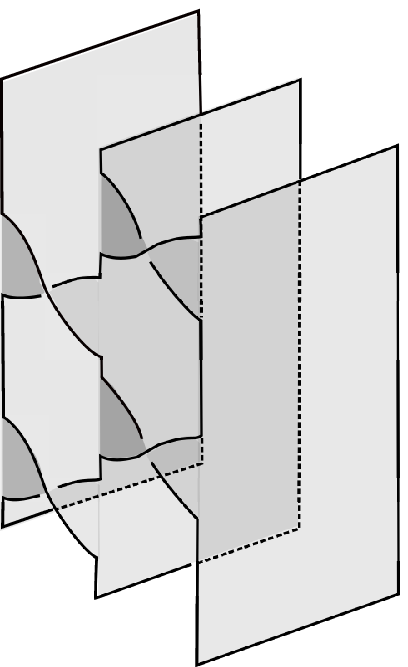 \end{tabular} \\
\end{tabular}}
\caption{(Left) A homogeneous braid $B$. (Right) A fiber of $\widehat{B}$ obtained by plumbing.}\label{fig_homogenous}
\end{figure}

The key to fiber stabilization is a theorem of Stallings which states that any link $J$ can be made into a fibered link by adding some unknotted component $J_0$ (see \cite{stallings}, Theorem 3). We now explain the idea behind Stallings' construction in detail. Let $J$ be any link and let $B$ be an $m$-strand braid word satisfying $\widehat{B} \leftrightharpoons J$. The exponent of the first occurrence of each $\sigma_i$ in $B$ will be considered as the \emph{correct sign of} $\sigma_i$. To correct the incorrect signs and make a fibered link, we will add strands to $B$ so that the homogeneity condition is enforced.

Suppose that there is some first letter $\alpha=\sigma_i^{\varepsilon}$ in $B$ with an incorrect sign. Add an $(m+1)$-st strand to the braid so that it dips around $\alpha$ to the left and returns to the $m+1$ level. This has the effect of pushing $\alpha$ over to the right by one strand. Thus, $\alpha$ now corresponds to the generator $\sigma_{i+1}^{\varepsilon}$. See the top picture in Figure \ref{fig_how_to_fiber}. Crossings of the new strand (drawn green) with the old strands are chosen so that they have the correct sign for each $\sigma_i$.  As $\sigma_m$ does not occur in $B$, we take $\varepsilon$ to be the correct sign of $\sigma_m$. If $\sigma_{i+1}^{\varepsilon}$ also has the incorrect sign, then we can continue to add parallel strands until the sign is correct or until it becomes $\sigma_{m}^{\varepsilon}$. As this now has the correct sign and no new incorrect $\sigma_i$ have been added, the total number of incorrect signs in $B$ has decreased by $1$.

Continuing inductively, we use the additional strands to remove the remaining incorrect crossings of $B$. More strands can be added to the right of the braid if necessary. If $k$ strands are added, then it is necessary to choose the correct sign for the generators $\sigma_{k+1},\ldots,\sigma_{m+k-1}$. Choose their correct sign to be $-{\varepsilon}$, where $\varepsilon$ is the sign of the first incorrect crossing as above. This choice ensures that it is possible to push any crossing over far enough so that it will eventually have the correct sign. Note that the closure of the $k$ added strands forms a $k$-component unlink. Concatenating the new braid with $b=\sigma_{m+1}^{-\varepsilon} \sigma_{m+2}^{-\varepsilon} \cdots\sigma_{m+k-1}^{-\varepsilon}$ gives a $k$-strand braid whose closure is the unknot. For a schematic diagram of the full construction, see Figure \ref{fig_how_to_fiber}, bottom. The closure of this braid is the desired fibered link $J_0 \sqcup J$.    

\begin{figure}[htb]
\fcolorbox{black}{white}{
\begin{tabular}{ccccc} \\
\begin{tabular}{c} \\ \def\svgwidth{.5in} 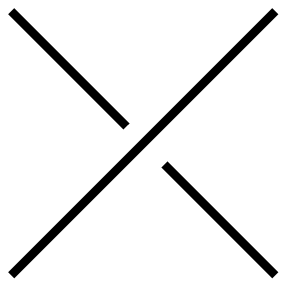 \end{tabular} & & \begin{tabular}{c} $\longrightarrow$ \end{tabular} & & \begin{tabular}{c}  \def\svgwidth{2in} 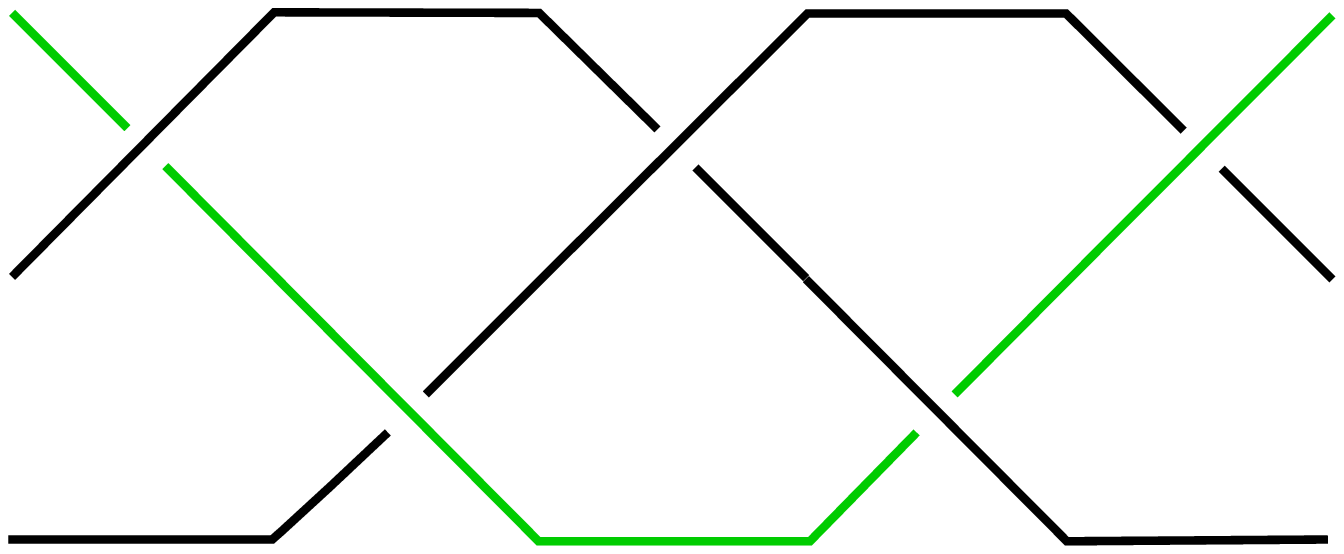 \end{tabular} \\ \\ \hline & & & & \\ \\
\multicolumn{5}{c}{\def\svgwidth{3.5in}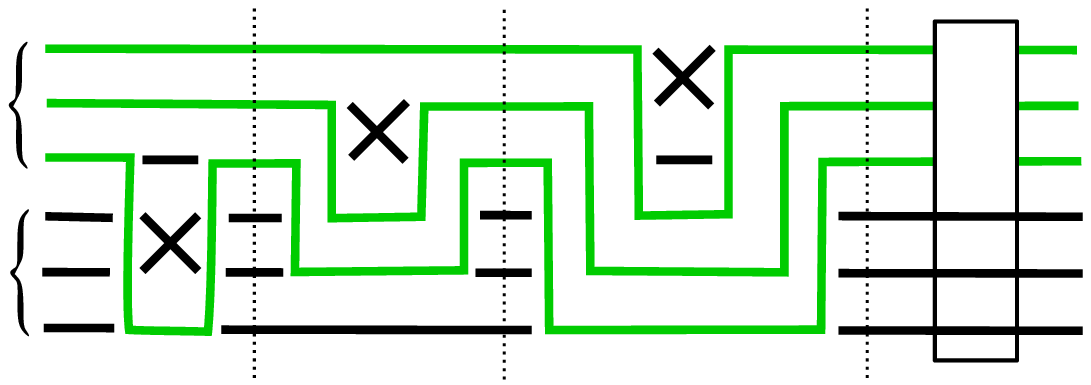}
\end{tabular}}
\caption{(Top) Adding a strand to force homogeneity. (Bottom) Schematic of adding a green unknotted component $J_0$ so that $J_0 \sqcup J$ is fibered.}\label{fig_how_to_fiber}
\end{figure}

\subsection{Proof of Theorem \ref{thm_fsl}} \label{sec_stable_proofs} With these preliminaries in place we are now ready to prove Theorem \ref{thm_fsl}: that every multi-component link has a fiber stabilization (Lemma \ref{lemma_fsl_1}) and that every fiber stabilized link as a virtual cover (Lemma \ref{lemma_fsl_2}).

\begin{lemma} \label{lemma_fsl_1} Every multi-component link $L=J \sqcup K$,with $K$ a knot, has a fiber stabilization.
\end{lemma}
\begin{proof} Represent $L$ as the closure of a parted mixed braid $B \sqcup \beta$, where the fixed part $B$ satisfies $\widehat{B} \leftrightharpoons J$ and the moving part $\beta$ satisfies $\widehat{\beta} \leftrightharpoons K$. Suppose that $B$ has $m$ black strands and $\beta$ has $n$ red strands. By the above remarks, there is an $m+k$ strand braid $B'$ such that $\widehat{B'}$ is a fibered link $J_0\sqcup J$ where $J_0$ is an unknot. Now replace the fixed part of $B \sqcup \beta$ with $B' \sqcup \beta$ as follows. The $k$ added strands (drawn in green) to $B$ are placed in between the $m$ strands of $B$ and the $n$ strands of $\beta$. In the top portion of the parted combed braid, we now have that the first $m+k$ strands are trivial. The crossings of the moving and the fixed part are unchanged. The moving part is chosen to always cross over any of the green strands. 

The previous section showed how to construct a fiber $\Sigma_0$ for $J_0 \sqcup J=\widehat{B'}$. It consists of discs attached to the closure of the $m+k$ strand trivial braid and half-twisted bands at the crossings. Then the algebraic intersection number of $\Sigma_0$ and $K$ can be made $0$ by adding an appropriate number of $(\pm)$ full twists between the leftmost red strand and the rightmost green strand. This construction is illustrated in Figure \ref{fig_comb_with_fiber}. This completes the proof.
\end{proof}

\begin{figure}[htb]
\fcolorbox{black}{white}{
\begin{tabular}{c}
\def\svgwidth{2.4in}
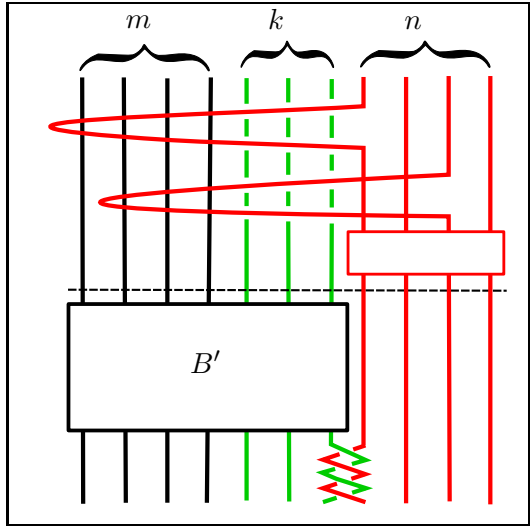 
\end{tabular}}
\caption{A fiber stabilization of a link $L=J \sqcup K$. The $m$ strands close to $J$, the $k$ strands close to an unknot $J_0$, and the $n$ strands close to $K$.} \label{fig_comb_with_fiber}
\end{figure}
   
\begin{lemma} \label{lemma_fsl_2} Every fiber stabilized link has a virtual cover.
\end{lemma}
\begin{proof} Let $L=J \sqcup K$ be an $(n+1)$-component link with $K$ a knot. Let $J_0 \subset \mathbb{S}^3$ be an unknot in the complement of $L$ such that $L_0 =J_0 \sqcup J \sqcup K$ is a fiber stabilization of $L$. Let $\Sigma_0$ be a fiber of $J_0 \sqcup J$ such that $K\cdot \Sigma_0=0$. There is a fiber bundle projection $p:N_0 \to \mathbb{S}^1$ with fiber $\Sigma_0$, where $N_0=\mathbb{S}^3\smallsetminus\nu(J_0 \sqcup J)$, such that the following diagram commutes.
\[
\xymatrix{\Sigma_0 \times \mathbb{R} \ar[r] \ar[d]_{\Pi}& \mathbb{R} \ar[d]^{\exp} \\
N_0 \ar[r]_p & \mathbb{S}^1}
\]
By hypothesis, $p_{\#}:\pi_1(N_0) \to \pi_1(\mathbb{S}^1)$ maps $[K]$ to $0 \in \mathbb{Z}$. Thus, $K$ lifts to a knot $\mathfrak{k}$ in $\Sigma_0 \times \mathbb{R}$ and there is a virtual cover $(\mathfrak{k}^{\Sigma_0 \times \mathbb{R}}, \Pi, K^{N_0})$.
\end{proof}

\section{Future directions} \label{sec_future}

Above we gave two instances of virtual knot invariants that are realizable by link invariants in $\mathbb{S}^3$. It is likely this program can be carried out for many other polynomial link invariants. Here we suggest some directions that seem promising for extending this point of view.
\begin{enumerate}
\item Understand index and triple linking numbers in a more general setting. There are examples where the index does not recover the triple linking number when the fiber is not a connected sum of trefoils or figure eight knots. We hope that this is a special case of a larger theorem that relates index in general to a geometrically defined quantity. 
\item Use virtual covers to compute Kauffman bracket skein modules for fibered knot complements. A particularly interesting focus would be hyperbolic fibered knot complements, such as the figure eight knot complement.
\item Section \ref{sec_alex} gave a formula relating MVAP of boundary links, which comes from the second elementary ideal, to the Alexander polynomial of AC knots, which comes from the first elementary ideal. For general SSF links, i.e. those that are not necessarily boundary links, is there a similar formula relating the MVAP of the first elementary ideal to the Alexander polynomial of the zeroth elementary ideal for virtual knots?
\item The finite-type concordance invariants of string links can be identified with the Milnor invariants \cite{habegger_masbaum}. The Henrich-Turaev polynomial \cite{henrich2010sequence} is a degree one finite-type concordance invariant of long virtual knots that is defined via the index, which is in turn given by the triple linking number. Can all finite-type concordance invariants of long virtual knots be represented as Milnor invariants of string links? 
\item  Can signatures for virtual knots be obtained by specialization of the multi-variable signature function of a multi-component link \cite{cimasoni_florens, cooper1982signatures,cooper1982universal}?
\item To what extent does a fiber stabilization produce a unique associated virtual knot? Suppose $L$ is a link and $L_0', L_0''$ are fiber stabilizations. If $L_0',L_0''$ are SSF links with invariant associated virtual knots $\upsilon'$, $\upsilon''$, respectively, what is the relationship between $\upsilon'$ and $\upsilon''$?
\end{enumerate}
We hope these questions will be considered in future papers and invite the community to participate in their development.
\newline
\newline
\noindent
\subsection*{Acknowledgements} The first author acknowledges a Spring 2015 Creativity and Research Grant from Monmouth University. These funds allowed him to visit the second author at the University of Nebraska-Omaha and to complete Section \ref{sec_index}. Both authors are grateful for helpful discussions and encouragement from C. Frohman and H. U. Boden. Lastly we would like to thank an anonymous referee for a close reading of a previous draft of this paper that led to an improvement in the statement of Theorem \ref{thm_main_alex}. The referee also improved the paper by pointing out some sign errors in the computations of $\bar{\mu}_{123}(L)$ and a computational error in the original arXiv version (arXiv:1706.07756v1[math.GT], Section 2.2) of this paper.

\bibliographystyle{acm}
\bibliography{IndexMVAP}

\begin{thebibliography}{10}

\bibitem{boden2015alexander}
{\sc Boden, H.~U., Dies, E., Gaudreau, R., Gerlings, A., Harper, E., and Nicas,
  A.~J.}
\newblock Alexander invariants for virtual knots.
\newblock {\em J. Knot Theory Ramifications 24}, 03 (2015), 1550009. \mrev{3342135} \zbl{06882268}

\bibitem{boden2015virtual}
{\sc Boden, H.~U., Gaudreau, R., Harper, E., Nicas, A.~J., and White, L.}
\newblock Virtual knot groups and almost classical knots.
\newblock {\em Fund. Math. 238}, 2 (2017), 101--142. \mrev{3640614} \zbl{1386.57007}

\bibitem{bz}
{\sc Burde, G., and Zieschang, H.}
\newblock {\em Knots}, second~ed., vol.~5 of {\em De Gruyter Studies in
  Mathematics}.
\newblock Walter de Gruyter \& Co., Berlin, 2003. \mrev{3156509} \zbl{1009.57003}

\bibitem{cairnselton}
{\sc Cairns, G., and Elton, D.~M.}
\newblock The planarity problem for signed {G}auss words.
\newblock {\em J. Knot Theory Ramifications 2}, 4 (1993), 359--367. \mrev{1247573} \zbl{0821.57010}

\bibitem{carter1991classifying}
{\sc Carter, J.~S.}
\newblock Classifying immersed curves.
\newblock {\em Proc. Amer. Math. Soc. 111}, 1 (1991), 281--287. \mrev{1043406} \zbl{ 0742.57008}

\bibitem{cheng}
{\sc Cheng, Z.}
\newblock A polynomial invariant of virtual knots.
\newblock {\em Proc. Amer. Math. Soc. 142}, 2 (2014), 713--725. \mrev{3134011} \zbl{1283.57010}

\bibitem{chrisman2014virtual}
{\sc Chrisman, M.}
\newblock Virtual covers of links.
\newblock {\em J. Knot Theory Ramifications 25}, 8 (2016), 1650052, 14. \mrev{3656323} \zbl{06617509}

\bibitem{chrisman2014three}
{\sc Chrisman, M., and Dye, H.~A.}
\newblock The three loop isotopy and framed isotopy invariants of virtual
  knots.
\newblock {\em Topology Appl. 173\/} (2014), 107--134. \mrev{3227210} \zbl{1311.57016}

\bibitem{chrisman2015virtual}
{\sc Chrisman, M., and Kaestner, A.}
\newblock Virtual covers of links {II}.
\newblock In {\em Knots, links, spatial graphs, and algebraic invariants},
  vol.~689 of {\em Contemp. Math.} Amer. Math. Soc., Providence, RI, 2017,
  pp.~65--80. \mrev{3656323} \zbl{06859865}

\bibitem{chrisman2013fibered}
{\sc Chrisman, M., and Manturov, V.~O.}
\newblock Fibered knots and virtual knots.
\newblock {\em J. Knot Theory Ramifications 22}, 12 (2013), 1341003. \mrev{3149309} \zbl{1325.57003}

\bibitem{cimasoni_florens}
{\sc Cimasoni, D., and Florens, V.}
\newblock Generalized {S}eifert surfaces and signatures of colored links.
\newblock {\em Trans. Amer. Math. Soc. 360}, 3 (2008), 1223--1264. \mrev{2357695} \zbl{1132.57004}

\bibitem{ct}
{\sc Cimasoni, D., and Turaev, V.}
\newblock A generalization of several classical invariants of links.
\newblock {\em Osaka J. Math. 44}, 3 (2007), 531--561. \mrev{2360939} \zbl{1148.57005}

\bibitem{cochran1985geometric}
{\sc Cochran, T.~D.}
\newblock Geometric invariants of link cobordism.
\newblock {\em Comment. Math. Helv. 60}, 1 (1985), 291--311. \mrev{800009} \zbl{0574.57008}

\bibitem{cochran1990derivatives}
{\sc Cochran, T.~D.}
\newblock {\em Derivatives of links: Milnor's concordance invariants and
  Massey's products}, vol.~427.
\newblock American Mathematical Soc., 1990. \mrev{1042041} \zbl{0705.57003}

\bibitem{cooper1982signatures}
{\sc Cooper, D.}
\newblock {\em Signatures of surfaces in 3-manifolds and applications to knot
  and link cobordism.}
\newblock PhD thesis, University of Warwick, 1982. 

\bibitem{cooper}
{\sc Cooper, D.}
\newblock The universal {A}belian cover of a link.
\newblock In {\em Low-dimensional topology ({B}angor, 1979)}, vol.~48 of {\em
  London Math. Soc. Lecture Note Ser.} Cambridge Univ. Press, Cambridge-New
  York, 1982, pp.~51--66. \mrev{662427} \zbl{0483.57004}



\bibitem{friedl2006algorithm}
{\sc Friedl, S.}
\newblock Algorithm for finding boundary link seifert matrices.
\newblock {\em J. Knot Theory Ramifications 15}, 05 (2006), 601--612. \mrev{2229330} \zbl{1100.57007}

\bibitem{vtable}
{\sc Green, J.}
\newblock A table of virtual knots.
\newblock {\em \url{http://www.math.toronto.edu/drorbn/Students/GreenJ}\/}
  (2004).

\bibitem{gutierrez}
{\sc Guti\'errez, M.~A.}
\newblock Polynomial invariants of boundary links.
\newblock {\em Rev. Colombiana Mat. 8\/} (1974), 97--109. \mrev{0367969} \zbl{0292.55004}

\bibitem{habegger_masbaum}
{\sc Habegger, N., and Masbaum, G.}
\newblock The {K}ontsevich integral and {M}ilnor's invariants.
\newblock {\em Topology 39}, 6 (2000), 1253--1289. \mrev{1783857} \zbl{0964.57011}

\bibitem{henrich2010sequence}
{\sc Henrich, A.}
\newblock A sequence of degree one {V}assiliev invariants for virtual knots.
\newblock {\em J. Knot Theory Ramifications 19}, 04 (2010), 461--487. \mrev{2646641} \zbl{1195.57030}

\bibitem{kauffman1998virtual}
{\sc Kauffman, L.~H.}
\newblock Virtual knot theory.
\newblock {\em European J. Combin. 20}, 7 (1999), 663--690. \mrev{1721925} \zbl{0938.57006}

\bibitem{kauffman2012introduction}
{\sc Kauffman, L.~H.}
\newblock Introduction to virtual knot theory.
\newblock {\em J. Knot Theory Ramifications 21}, 13 (2012), 1240007. \mrev{3021769} \zbl{1255.57005}

\bibitem{kawauchi}
{\sc Kawauchi, A.}
\newblock {\em A survey of knot theory}.
\newblock Birkh\"auser Verlag, Basel, 1996.
\newblock Translated and revised from the 1990 Japanese original by the author. \mrev{1417494} \zbl{0861.57001}

\bibitem{kuperberg2003virtual}
{\sc Kuperberg, G.}
\newblock What is a virtual link?
\newblock {\em Algebr. Geom. Topol. 3}, 1 (2003), 587--591. \mrev{1997331} \zbl{1031.57010}

\bibitem{lamb_rourke}
{\sc Lambropoulou, S., and Rourke, C.~P.}
\newblock Algebraic {M}arkov equivalence for links in three-manifolds.
\newblock {\em Compos. Math. 142}, 4 (2006), 1039--1062. \mrev{2249541} \zbl{1031.57010}

\bibitem{vktsoa}
{\sc Manturov, V.~O., and Ilyutko, D.~P.}
\newblock {\em Virtual knots: the state of the art}, vol.~51 of {\em Series on
  Knots and Everything}.
\newblock World Scientific Publishing Co. Pte. Ltd., Hackensack, NJ, 2013.
\newblock Translated from the 2010 Russian original, With a preface by Louis H.
  Kauffman. \mrev{} \zbl{1270.57003}

\bibitem{mellor2003geometric}
{\sc Mellor, B., and Melvin, P.}
\newblock A geometric interpretation of {M}ilnor's triple linking numbers.
\newblock {\em Algebr. Geom. Topol. 3}, 1 (2003), 557--568. \mrev{1997329} \zbl{1040.57007}

\bibitem{milnor1957isotopy}
{\sc Milnor, J.}
\newblock Isotopy of links. {A}lgebraic geometry and topology.
\newblock In {\em A symposium in honor of {S}. {L}efschetz}. Princeton
  University Press, Princeton, N. J., 1957, pp.~280--306. \mrev{0092150} \zbl{0080.16901}

\bibitem{rolfsen1976knots}
{\sc Rolfsen, D.}
\newblock {\em Knots and links}, vol.~346.
\newblock American Mathematical Soc., 1976. \mrev{0515288} \zbl{0339.55004}

\bibitem{saw}
{\sc Sawollek, J.}
\newblock On {A}lexander-{C}onway polynomials for virtual knots and links.
\newblock {\em arXiv:9912173[math.GT]\/} (December 1999). \mrev{} \zbl{}

\bibitem{silwill_AC}
{\sc Silver, D.~S., and Williams, S.~G.}
\newblock Crowell's derived group and twisted polynomials.
\newblock {\em J. Knot Theory Ramifications 15}, 8 (2006), 1079--1094. \mrev{2275098} \zbl{1115.57006}

\bibitem{stall_certain}
{\sc Stallings, J.}
\newblock On fibering certain {$3$}-manifolds.
\newblock In {\em Topology of 3-manifolds and related topics ({P}roc. {T}he
  {U}niv. of {G}eorgia {I}nstitute, 1961)}. Prentice-Hall, Englewood Cliffs,
  N.J., 1962, pp.~95--100. \mrev{0158375} \zbl{1246.57049}

\bibitem{stallings}
{\sc Stallings, J.}
\newblock Constructions of fiberd knots and links.
\newblock In {\em Algebraic and geometric topology ({P}roc. {S}ympos. {P}ure
  {M}ath., {S}tanford {U}niv., {S}tanford, {C}alif., 1976), {P}art 2}, Proc.
  Sympos. Pure Math., XXXII. Amer. Math. Soc., Providence, R.I., 1978,
  pp.~55--60. \mrev{520516} \zbl{0385.00008}

\bibitem{whitten_sublinks}
{\sc Whitten, Jr., W.~C.}
\newblock On noninvertible links with invertible proper sublinks.
\newblock {\em Proc. Amer. Math. Soc. 26\/} (1970), 341--346. \mrev{0264647} \zbl{0198.56503}

\end{thebibliography}

\end{document}